\documentclass[11pt,oneside,reqno]{amsart}

\usepackage{hyperref}
\usepackage{amsxtra}
\usepackage{amsopn}
\usepackage{amsmath,amsthm,amssymb}
\usepackage{color}
\usepackage{amscd}
\usepackage{pifont}
\usepackage{amsfonts}
\usepackage{latexsym}
\usepackage{verbatim}
\usepackage{pb-diagram}
\usepackage{tikz}

\newcommand{\fr}{\mathfrak}

 \newtheorem{lemma} {Lemma} [section]
\newtheorem{theorem}[lemma]{Theorem} 
\newtheorem{remark}[lemma] {Remark} 
\newtheorem{prop} [lemma]{Proposition}  
\newtheorem{definition}[lemma] {Definition} 
\newtheorem{corol}[lemma] {Corollary} 
\newtheorem{example}[lemma] {Example}

\newtheorem{claim}[lemma] {Claim}
\newtheorem{assumption}[lemma]{Assumption}

\begin{document}
\title{On a class of geodesic orbit spaces with abelian isotropy subgroup}
\author{Nikolaos Panagiotis Souris}
\address{University of Patras, Department of Mathematics, University Campus, 26504, Rio Patras, Greece}
\email{nsouris@upatras.gr}

\begin{abstract}
Riemannian geodesic orbit spaces $(G/H,g)$ are natural generalizations of symmetric spaces, defined by the property that their geodesics are orbits of one-parameter subgroups of $G$.  We study the geodesic orbit spaces of the form $(G/S,g)$, where $G$ is a compact, connected, semisimple Lie group and $S$ is abelian.  We give a simple geometric characterization of those spaces, namely that they are naturally reductive. In turn, this yields the classification of the invariant geodesic orbit (and also the naturally reductive) metrics on any space of the form $G/S$.  Our approach involves simplifying the intricate parameter space of geodesic orbit metrics on $G/S$ by reducing their study to certain submanifolds and generalized flag manifolds, and by studying properties of root systems of simple Lie algebras associated to these manifolds.



\end{abstract}
\maketitle

\section{Introduction}

The classification of Riemannian symmetric spaces by E. Cartan has stimulated the study of several classes of Riemannian manifolds that generalize notable properties of symmetric spaces.  Examples include the classes of \emph{isotropy irreducible spaces} (\cite{Wo1}), \emph{weakly symmetric spaces} (\cite{Sel}, \cite{Wo2}), \emph{$\delta$-homogeneous spaces} (\cite{BeNi1}) and \emph{Clifford-Wolf homogeneous spaces} (\cite{BeNi2}).  All the aforementioned manifolds $(M,g)$ share the property that their geodesics $\gamma$ are orbits of one-parameter groups of isometries, or equivalently, there exists a group $G$ of isometries of $(M,g)$ such that 

\begin{equation*}\label{1}\gamma(t)=\exp(tX)\cdot o,\end{equation*}

\noindent where $\exp$ is the exponential map on $G$, $o\in M$ and $\cdot$ denotes the action of $G$ on $M$.  Manifolds with this property are called \emph{geodesic orbit manifolds} (or \emph{g.o. manifolds}) and are extensively studied for the last thirty years within the Riemannian, pseudo-Riemannian and Finsler geometric context (see the recent studies/surveys \cite{Ar2}, \cite{CaZa}, \cite{CheCheWo}, \cite{GoNi}, \cite{Ni}, \cite{YaDe} and the references therein).  Any g.o. manifold $(M,g)$ is homogeneous, that is $M=G/H$, where $H$ is the (closed) isotropy subgroup of a point in $M$.  The corresponding Riemannian space $(G/H,g)$ is called a \emph{g.o. space} and the $G$-invariant (invariant by the action of $G$) metric $g$ is called a \emph{g.o. metric}.

The most common examples of g.o. metrics are the \emph{naturally reductive metrics} (Definition \ref{NatRed}). These are essentially the metrics induced from bi-invariant (pseudo-)Riemannian metrics of Lie subgroups of $G$ acting transitively on $G/H$ (see for example \cite{DaZi}, \cite{Kos} or Theorem \ref{KosThe}). On the other hand, there exist non-naturally reductive g.o. metrics (\cite{Ka}) but they are much more rare than their naturally reductive counterparts. The prime examples of naturally reductive metrics on compact spaces are the \emph{normal metrics} (Definition \ref{normal}), namely those metrics induced from bi-invariant Riemannian metrics on $G$. 

Riemannian naturally reductive spaces have been previously studied in \cite{DaZi}, \cite{Go1}, \cite{TriVa}, among several other works. Their classification is open, while interesting structural and classification results were recently obtained in the context of metric connections with skew-symmetric torsion (\cite{AgFeFr}, \cite{OlReTa}, \cite{St1}, \cite{St2}).  As is the case with naturally reductive spaces, the full classification of g.o. manifolds remains an open question. On the other hand, several partial classifications have been obtained (\cite{AlNi}, \cite{Go2}, \cite{KovVa}, \cite{Nik0}, \cite{Ta}), incuding the classification of the g.o. flag manifolds (\cite{AlAr}) and the recent classification of the g.o. spaces with two isotropy summands (\cite{CheNi}). 

The purpose of this paper is to study the class of g.o. spaces of the form $(G/S,g)$, where $G$ is a compact, connected, semisimple Lie group and the isotropy subgroup $S$ is abelian. Since the possible embeddings of closed abelian subgroups in semisimple Lie groups are varied (for example, non-maximal tori of a given dimension may not be conjugate), the spaces $G/S$ form quite a large class. An important aspect of this class is that several of its elements arise in partial classifications of g.o. spaces (e.g. \cite{AlAr}) or when one seeks to classify g.o. spaces or naturally reductive spaces of a given dimension (\cite{AgFeFr}, \cite{St2}).  Our result is the following simple characterization.

\begin{theorem}\label{main}
Let $(G/S,g)$ be a Riemannian homogeneous space where $G$ is a compact, connected, semisimple Lie group and $S$ is abelian.  Then $(G/S,g)$ is a geodesic orbit space if and only if the metric $g$ is ($G$-)naturally reductive.  In particular, $(G/S,g)$ is a geodesic orbit space if and only if the metric $g$ is normal, that is, $g$ is induced from a bi-invariant Riemannian metric on $G$.
\end{theorem} 

\noindent Since the bi-invariant Riemannian metrics are in bijection with $\operatorname{Ad}$-invariant inner products and the latter are explicitly known for compact semisimple Lie algebras (Proposition \ref{DZp}), a consequence of Theorem \ref{main} is the classification of the invariant Riemannian geodesic orbit (and also the naturally reductive) metrics on any space of the form $G/S$.  The classification is given as follows.

\begin{corol}\label{Class}

Let $G/S$ be a homogeneous space where $G$ is a compact, connected, semisimple Lie group and $S$ is abelian. Let $\fr{g},\fr{s}$ be the Lie algebras of $G,S$ respectively, let $B$ denote the negative of the Killing form of $\fr{g}$ and consider the decomposition $\fr{g}=\bigoplus_{j=1}^k\fr{g}_j$ of $\fr{g}$ into simple ideals.  Then there is a one to one correspondence between \emph{(i)} Riemannian $G$-invariant geodesic orbit metrics $g$ on $G/S$ and \emph{(ii)} restrictions to $\widetilde{\fr{m}}$ of inner products of the form $Q=\sum_{j=1}^k\left.\mu_j B\right|_{\fr{g}_j\times \fr{g}_j}$, $\mu_j>0$ (up to homothety), where $\widetilde{\fr{m}}$ is the $Q$ orthogonal complement of $\fr{s}$ in $\fr{g}$.
\end{corol}

\noindent When $G$ is simple, the only $\operatorname{Ad}$-invariant inner products on its Lie algebra are the scalar multiples of the Killing form.  As a result, we have the following.

\begin{corol}Let $G/S$ be a homogeneous space where $G$ is a compact, connected, simple Lie group and $S$ is abelian.  A Riemannian $G$-invariant metric $g$ on $G/S$ is geodesic orbit if and only if it is the standard metric induced from the Killing form.
\end{corol}

 \noindent Finally, Theorem \ref{main} adds the spaces $G/S$ to the class of spaces whose invariant g.o. metrics are necessarily naturally reductive.  The latter class includes all compact \emph{Lie groups} (in fact any left-invariant g.o. metric on a compact Lie group is necessarily bi-invariant), most \emph{compact, simply connected spaces of positive Euler characteristic} (\cite{AlNi}) and the \emph{Ledger-Obata} spaces (\cite{NikNi}).
     
\subsection{Overview of the approach}\label{overview}

To establish the natural reductivity of the g.o. metrics on $G/S$, our goal is for each such metric to show the linearity of a certain map $\xi$, called the \emph{geodesic graph}, whose domain is the tangent space $\fr{m}$ at the origin of $G/S$ and its image lies in the Lie algebra of $S$ (Proposition \ref{SZ2}).  In general, the parameter space of the invariant metrics on a homogeneous space $G/H$ depends on the so-called \emph{isotropy representation} of $H$ on the tangent space at the origin, and as a rule, the smaller the isotropy subgroup $H$ is, the more complicated is the parameter space of metrics (for further details see the related discussion in \cite{Wa}, p. 255-256, or the explicit description of metrics in terms of the isotropy representation in \cite{So}). 
 
In particular, the parameter space of metrics on $G/S$ is intricate because the isotropy representation of $S$ contains numerous equivalent submodules (in fact, since $S$ is abelian, any real irreducible submodule is either one or two dimensional).  More specifically, the large number of the irreducible submodules implies that the parameter space of metrics has large dimension, while the equivalence of the submodules, apart from further contributing to the increase of the dimension, also implies that the metrics have no ``obvious" diagonal form.  

We simplify the parameter space of metrics on $G/S$ in Section \ref{Reduction} by firstly assuming that $S$ is a torus, and then by observing, using a result of Alekseevsky and Nikonorov in \cite{AlNi}, that any g.o. metric on $G/S$ induces a g.o. metric on the \emph{flag manifold} $G/C_G(S)$ and a bi-invariant metric on the Lie group $C_G(S)/S$, where $C_G(S)$ denotes the \emph{centralizer} of $S$ in $G$  (in fact, the tangent space of the flag manifold coincides with the sum of the two-dimensional submodules of $S$ while the Lie algebra of $C_G(S)/S$ coincides with the sum of the one-dimensional submodules).  A large extent of our subsequent work then relies on the classification of the g.o. metrics on flag manifolds by Alekseevski and Arvanitoyeorgos in \cite{AlAr} and on the Lie-theoretic description of those manifolds, which we provide in Section \ref{G/K}.    

Apart from the aforementioned reduction, we use simplification techniques for g.o. metrics established in \cite{AlNi} and \cite{So}, while we derive new results in this direction using Lie-theoretic arguments, and in particular properties of root systems of simple Lie algebras and their subsystems associated to flag manifolds (subsection \ref{tech}).  As a result, in Section \ref{mainresults} we obtain a simple, necessary form for the g.o. metrics on $G/S$ (Theorem \ref{NecForm}), which leads to the proof of the first part of Theorem \ref{main} in Section \ref{proof}.  To obtain the second part of Theorem \ref{main}, we apply Kostant's characterization of naturally reductive spaces (Theorem \ref{KosThe}).  In Section \ref{proof} we also prove Corollary \ref{Class}.

\section{Preliminaries}\label{Prel}

\subsection{Invariant metrics and the metric endomorphism in $G/H$}\label{InvMet}

Let $G/H$ be a homogeneous space and let $o=eH$ be the \emph{origin} of $G/H$.  Denote by $\fr{g},\fr{h}$ the Lie algebras of $G,H$ respectively.  Moreover, let $\operatorname{Ad}:G\rightarrow \operatorname{Aut}(\fr{g})$ and $\operatorname{ad}:\fr{g}\rightarrow \operatorname{End}(\fr{g})$ be the \emph{adjoint representations} of $G$ and $\fr{g}$ respectively, where $\operatorname{ad}_XY=[X,Y]$.  We assume that $G/H$ is \emph{reductive}, that is there exists a decomposition

\begin{equation*}\label{dec}\fr{g}=\fr{h}\oplus \fr{m}.\end{equation*}

\noindent such that $\operatorname{Ad}_H\fr{m}\subseteq \fr{m}$.  Then $[\fr{h},\fr{m}]\subseteq \fr{m}$, and $\fr{m}$ can be naturally identified with the tangent space $T_o(G/H)$.  If $G$ is compact, the algebra $\fr{g}$ admits an $\operatorname{Ad}$-invariant inner product $B$ (in turn, any operator $\operatorname{ad}_X$, $X\in \fr{g}$, is skew symmetric with respect to $B$).  In such a case, $G/H$ is reductive and we may choose $\fr{m}$ to be the $B$-orthogonal complement of $\fr{h}$ in $\fr{g}$.

For $x\in G$, let $\tau_x:G/H\rightarrow G/H$ denote the (left) action of $x$ on $G/H$. A Riemannian metric $g$ on $G/H$ is called \emph{G-invariant} if the maps $\tau_x$ are isometries for all $x\in G$. A space $G/H$, equipped with a $G$-invariant Riemannian metric $g$, is called a \emph{Riemannian homogeneous space} and is denoted by $(G/H,g)$.  The $G$-invariant Riemannian metrics $g$ on $G/H$ are in one to one correspondence with $\operatorname{Ad}_H$-invariant inner products $g( \ ,\ )_o$ on $\fr{m}$.  In turn, if we fix an $\operatorname{Ad}_H$-invariant inner product $B$ on $\fr{m}$, the $G$-invariant Riemannian metrics $g$ on $G/H$ are in one to one correspondence with \emph{metric endomorphisms} $A:\fr{m}\rightarrow \fr{m}$, satisfying 

\begin{equation}\label{MetEnd}g( X,Y )_o=B(AX,Y),\ X,Y\in \fr{m}.
\end{equation}
\noindent  When $G$ is compact, we will always assume that the fixed product $B$ is an $\operatorname{Ad}$-invariant inner product on $\fr{g}$.  Moreover, if $G$ is semisimple, $B$ will denote the negative of the Killing form of $\fr{g}$.

Each metric endomorphism $A$ is symmetric with respect to $B$, positive definite, $\operatorname{Ad}_H$-equivariant and $\operatorname{ad}_{\fr{h}}$-equivariant (i.e. $A$ commutes with $\operatorname{Ad}_h$ for all $h\in H$ and with $\operatorname{ad}_{X}$ for all $X\in \fr{h}$).  Conversely, any endomorphism with the aforementioned properties corresponds to a unique $G$-invariant Riemannian metric on $G/H$ via Equation \eqref{MetEnd}.  Since $A$ is diagonalizable, it induces a $B$-orthogonal decomposition $\fr{m}=\bigoplus_{i=1}^n{\fr{m}_{\lambda_i}}$, where $\fr{m}_{\lambda_i}$ are the ($\operatorname{Ad}_H$-invariant) eigenspaces of $A$, corresponding to pairwise distinct eigenvalues $\lambda_i>0$.  The general form of a metric endomorphism is explicitly expressed in terms of the irreducible submodules of the isotropy representation $\chi:H\rightarrow \operatorname{Gl}(\fr{m})$, given by $\chi(h)X=\operatorname{Ad}_hX$ (see for example \cite{So} or \cite{Wa}), but we will not go into further details here.

\subsection{Geodesic orbit metrics, geodesic graphs and naturally reductive metrics}\label{GOMetrics}

\begin{definition}\label{GOMet} A $G$-invariant Riemannian metric $g$ on a homogeneous space $G/H$ is called a geodesic orbit metric (or a g.o. metric) if all geodesics of $(G/H,g)$ are orbits of one-parameter subgroups of $G$.  Equivalently, $g$ is a geodesic orbit metric if for any geodesic $\gamma$ of $(G/H,g)$ there exists a vector $X\in \fr{g}\setminus \{0\}$ such that $\gamma(t)=\exp(tX)\cdot o$.  The space $(G/H,g)$ is then called a geodesic orbit space (or g.o. space).
\end{definition} 

A left-invariant metric $g$ on a compact Lie group $G$ is a g.o. metric if and only if it is bi-invariant (\cite{AlNi}).  This is equivalent to the corresponding inner product $g( \ , \ )_e$ on $\fr{g}$ being $\operatorname{Ad}$-invariant.  If $G$ is connected, the bi-invariance of a metric $g$ is equivalent to every operator $\operatorname{ad}_X, X\in \fr{g}$, being skew-symmetric.  We also have the following.
 
\begin{lemma}\label{Bi}
A left-invariant metric $g$ on a compact, connected Lie group $G$ is bi-invariant if and only if the corresponding metric endomorphism $A$ satisfies $[X,AX]=0$ for all $X\in \fr{g}$.  
\end{lemma}

\begin{proof} 
Fix an $\operatorname{Ad}$-invariant inner product $B$ on $\fr{g}$.  The metric $g$ is bi-invariant if and only if $g( [X,Y],X)_e=0$ for all $X,Y\in \fr{g}$.  By virtue of Equation \eqref{MetEnd} and the $\operatorname{Ad}$-invariance of $B$, the last equation is equivalent to $0=B([X,Y],AX)= -B(Y,[X,AX])$ for all $X,Y\in \fr{g}$, which yields the desired result.\end{proof}

The following defines the most natural class of g.o. metrics, namely those induced from bi-invariant Riemannian metrics.

\begin{definition}\label{normal}
A $G$-invariant metric $g$ on a homogeneous space $G/H$ is called normal if there exists an $\operatorname{Ad}$-invariant inner product $Q$ on the Lie algebra $\fr{g}$ of $G$ such that 

\begin{equation*}g( \ , \ )_o=\left.Q\right|_{\fr{m}\times \fr{m}},\end{equation*}

\noindent where $\fr{m}$ is the $Q$-orthogonal complement of the Lie algebra $\fr{h}$ of $H$ in $\fr{g}$.
\end{definition}
 
Hence, every homogeneous space $G/H$ with $G$ compact admits at least one normal metric and thus at least one g.o. metric.  When $G$ is simple, the only normal metric $g$ on $G/H$ (up to homothety) is the one induced from a scalar multiple of the Killing form of $\fr{g}$. In that case, $g$ is called \emph{standard}. \\

The well-known \emph{geodesic lemma} of Kowalski and Vanhecke in \cite{KovVa} states that $(G/H,g)$ is a g.o. space if and only if there exists a map $\xi:\fr{m}\rightarrow \fr{h}$, called a \emph{geodesic graph}, such that  

\begin{equation}\label{Kow}g\big( [\xi(X)+X,Y]_{\fr{m}}, X\big)_o=0 \ \ \makebox{for all} \ \ X,Y\in \fr{m}.
\end{equation} 

The geodesic graph $\xi$ was introduced by Szenthe in \cite{Sz}. Its form roughly depends on the embedding of $H$ in $G$, and $\xi$ may be non-differentiable at zero.  We also remark that under the assumption that $(G/H,g)$ is a geodesic orbit space, a geodesic graph $\xi$ may not be unique.  The simplest case arises when $\xi\equiv 0$ leading to the following definition.

\begin{definition} \label{NatRed}
A Riemannian homogeneous space $(G/H,g)$ is called naturally reductive with respect to the reductive decomposition $\fr{g}=\fr{h}\oplus \fr{m}$ if  

\begin{equation*}g([X,Y]_{\fr{m}},Z)_o+g(Y,[X,Z]_{\fr{m}})_o=0 \ \ \makebox{for all} \ \ X,Y,Z\in \fr{m},\end{equation*}

 \noindent or equivalently, if Equation \eqref{Kow} holds with $\xi\equiv 0$. More generally, the space $(G/H,g)$ is called naturally reductive (or $G$-naturally reductive) if it is naturally reductive with respect to some reductive decomposition $\fr{g}=\fr{h}\oplus \widetilde{\fr{m}}$.  The metric $g$ is also called ($G$-)naturally reductive. 
 \end{definition}

The following is a useful criterion for naturally reductive spaces in terms of the graph $\xi$.

\begin{prop}\label{SZ1}\emph{(\cite{KovNik}, \cite{KovVa}), \cite{Sz})} 
A Riemannian homogeneous space $(G/H,g)$, with the reductive decomposition $\fr{g}=\fr{h}\oplus \fr{m}$, is ($G$-)naturally reductive with respect to some reductive decomposition $\fr{g}=\fr{h}\oplus \widetilde{\fr{m}}$ if and only if there exists an $\operatorname{Ad}_H$-equivariant linear map $\xi:\fr{m}\rightarrow \fr{h}$ such that condition \eqref{Kow} is true.  Moreover, in this case $\widetilde{\fr{m}}=\{\xi(X)+X:X\in \fr{m}\}$.
\end{prop}

Any normal metric on a space $G/H$ with $G$ compact is naturally reductive but the converse does not hold.  Moreover, any naturally reductive metric is geodesic orbit, but the converse also does not hold.  The prime example of a non-naturally reductive, geodesic orbit space is the generalized Heisenberg group (\cite{Ka}).\\

The action of $G$ on $G/H$ is called \emph{almost effective} if any subgroup of $H$ that is normal in $G$ is discrete.  The following characterization of Kostant states that, under an almost effectiveness condition, the naturally reductive metrics on $G/H$ are precisely those metrics induced from bi-invariant, possibly non-positive definite metrics of subgroups of $G$ acting transitively on $G/H$.

\begin{theorem}\label{KosThe}\emph{(\cite{DaZi}, \cite{Kos})}    
Let $G/H$ be a connected homogeneous space such that $G$ acts almost effectively on $G/H$.  Assume that $g$ is a $G$-invariant Riemannian metric on $G/H$ which is naturally reductive with respect to the decomposition $\fr{g}=\fr{h}\oplus \fr{m}$.  Then the space $\widetilde{\fr{g}}:=\fr{m}+[\fr{m},\fr{m}]$ is an ideal of $\fr{g}$ such that the corresponding connected subgroup $\widetilde{G}$ of $G$ acts transitively on  $G/H$, and there exists a unique $\operatorname{Ad}_{\widetilde{G}}$-invariant, symmetric, non-degenerate bilinear form $Q$ on $\widetilde{\fr{g}}$ (not necessarily positive definite) such that 

\begin{equation*}Q(\fr{h}\cap\widetilde{\fr{g}},\fr{m})=\{0\} \ \ \makebox{and}\ \ \left.Q\right|_{\fr{m}\times \fr{m}}=g( \ , \ )_o.\end{equation*}

\noindent Here $\fr{h}\cap\widetilde{\fr{g}}$ is the isotropy algebra in $\widetilde{\fr{g}}$.  Conversely, if $G$ is connected, then for any $\operatorname{Ad}_{G}$-invariant, symmetric, non-degenerate bilinear form $Q$ on $\fr{g}$, which is also non-degenerate on $\fr{h}$ and positive definite on $\fr{m}$, the metric on $G/H$ defined by $g( \ , \ )_o=\left.Q\right|_{\fr{m}\times \fr{m}}$ is naturally reductive.  
\end{theorem}

The following characterizes the $\operatorname{Ad}$-invariant, symmetric, non-degenerate bilinear forms on compact semisimple Lie algebras and thus the bi-invariant metrics on compact semisimple Lie groups (see for example \cite{DaZi}).

\begin{prop}\label{DZp}
Let $\fr{g}$ be a compact semisimple Lie algebra and let $B$ denote the negative of the Killing form of $\fr{g}$.  Consider the $B$-orthogonal decomposition $\fr{g}=\fr{g}_1\oplus \cdots \oplus \fr{g}_k$, of $\fr{g}$ into the simple ideals $\fr{g}_j$.  A bilinear form  $Q$ on $\fr{g}$ is $\operatorname{Ad}$-invariant, symmetric and non-degenerate if and only if  

\begin{equation}\label{DZppp}
Q=\left.\mu_1B\right|_{\fr{g}_1\times \fr{g}_1}+\cdots +\left.\mu_kB\right|_{\fr{g}_k\times \fr{g}_k} \ \ \makebox{where} \ \ \mu_j\in \mathbb R^*.
\end{equation}

Moreover, $Q$ is an inner product if and only if $\mu_j>0$, $j=1,\dots,k$.  Therefore, a left-invariant Riemannian metric $g$ on a compact semisimple Lie group $G$ is bi-invariant if and only if the corresponding product $g(\ , \ )_e$ on $\fr{g}$ has the form \eqref{DZppp} with $\mu_j>0$, $j=1,\dots ,k$.
\end{prop}

We close this subsection by turning our attention to some useful results for geodesic orbit metrics on compact spaces.  We assume that $G$ is a compact Lie group with Lie algebra $\fr{g}$, and we fix an $\operatorname{Ad}$-invariant inner product $B$ on $\fr{g}$ (if $\fr{g}$ is semisimple, $B$ will denote the negative of the Killing form).  Moreover, if $H$ is a closed subgroup of $G$ with Lie algebra $\fr{h}$, we fix a $B$-orthogonal reductive decomposition $\fr{g}=\fr{h}\oplus \fr{m}$. For the rest of the paper, we will make no distinction between a $G$-invariant metric $g$ on $G/H$ and its corresponding metric endomorphism $A:\fr{m}\rightarrow \fr{m}$ satisfying Equation \eqref{MetEnd}.  As a result, the notations $(G/H,g)$, $(G/H,A)$ will denote the same object.\\

The following is a necessary and sufficient algebraic condition for a $G$-invariant metric to be g.o., arising from condition \eqref{Kow}. 
    
\begin{prop}\label{GOCond}\emph{(\cite{AlAr}, \cite{So})} Let $G$ be a compact Lie group.  The Riemannian space $(G/H,A)$ is a g.o. space if and only if there exists a map $\xi:\fr{m}\rightarrow \fr{h}$ such that 

\begin{equation}\label{GO}[\xi(X)+X,AX]=0 \ \ \makebox{for all} \ \ X\in \fr{m}.\end{equation}
\end{prop}

\noindent Accordingly, Proposition \ref{SZ1} is formulated as follows for $G$ compact.

\begin{prop}\label{SZ2} 
Let $G$ be a compact Lie group and let $(G/H,g)$ be a Riemannian homogeneous space with the reductive decomposition $\fr{g}=\fr{h}\oplus \fr{m}$, and with corresponding metric endomorphism $A:\fr{m}\rightarrow \fr{m}$ satisfying Equation \eqref{MetEnd}. Then $(G/H,g)$ is ($G$-)naturally reductive with respect to some reductive decomposition $\fr{g}=\fr{h}\oplus \widetilde{\fr{m}}$ if and only if there exists an $\operatorname{Ad}_H$-equivariant linear map $\xi:\fr{m}\rightarrow \fr{h}$ such that condition \eqref{GO} is true.  Moreover, $\widetilde{\fr{m}}=\{\xi(X)+X:X\in \fr{m}\}$.
\end{prop}

For a g.o. space $(G/H,A)$, let $\fr{m}_{\lambda_1},\fr{m}_{\lambda_2}\subset \fr{m}$ be two eigenspaces of $A$ with respect to distinct eigenvalues $\lambda_1,\lambda_2$ respectively.  Assume that $\fr{m}_1,\fr{m}_2$ are $\operatorname{ad}_{\fr{h}}$-invariant subspaces such that $\fr{m}_1\subseteq\fr{m}_{\lambda_1}$ and $\fr{m}_2\subseteq\fr{m}_{\lambda_2}$.  Using Equation \eqref{GO} and the $\operatorname{ad}_{\fr{h}}$-invariance of $\fr{m}_{\lambda_1}$ and $\fr{m}_{\lambda_2}$, it is not hard to show that 

\begin{equation}\label{EiBra}[\fr{m}_1,\fr{m}_2]\subseteq \fr{m}_1\oplus \fr{m}_2,\end{equation}
  
\noindent (see for example \cite{AlNi}, Corollary 3).  The above property is quite useful; It allows one to simplify the possible form of the g.o. metrics by reducing the number of the distinct eigenvalues of the corresponding metric endomorphism.  In particular, we have the following result (see also (\cite{So}, Proposition 5).

\begin{lemma}\label{EigenEq}
 Let $(G/H,A)$ be a g.o. space with $G$ compact.  Assume that $\lambda_1,\lambda_2$ are eigenvalues of $A$ and that there exist $\operatorname{ad}_{\fr{h}}$-invariant subspaces $\fr{m}_1,\fr{m}_2$ of $\fr{m}$ satisfying the following:\\
  \emph{{1.}} $\left.A\right|_{\fr{m}_1}=\lambda_1\operatorname{Id}$ and $\left.A\right|_{\fr{m}_2}=\lambda_2\operatorname{Id}$.\\
   \emph{{2.}} The space $[\fr{m}_1,\fr{m}_2]$ has non-zero projection on $(\fr{m}_1+\fr{m}_2)^{\bot}$.\\
    Then $\lambda_1=\lambda_2$. 
\end{lemma}

\begin{proof} Let $\fr{m}_{\lambda_1}, \fr{m}_{\lambda_2}$ be the corresponding eigenspaces of $\lambda_1,\lambda_2$.  If $\lambda_1\neq \lambda_2$, then condition {1.} in the hypothesis yields $\fr{m}_i\subseteq \fr{m}_{\lambda_i}$, $i=1,2$.  Relation \eqref{EiBra} then yields $[\fr{m}_1,\fr{m}_2]\subseteq \fr{m}_1\oplus \fr{m}_2$, which contradicts condition {2.} in the hypothesis.\end{proof}

\subsection{Totally geodesic orbits and induced geodesic orbit spaces}\label{IndGO}

Let $(G/H,g)$ be a Riemannian g.o. space with $G$ compact, and let $K$ be a closed subgroup of $G$ containing $H$.  Let $\fr{h},\fr{k}$ be the Lie algebras of $H,K$ respectively and consider the $B$-orthogonal decompositions $\fr{k}=\fr{h}\oplus \fr{p}$ and $\fr{g}=\fr{k}\oplus \fr{q}=\fr{h}\oplus \fr{p}\oplus \fr{q}$, so that decomposition $\fr{g}=\fr{h}\oplus \fr{m}$ yields $\fr{m}=\fr{p}\oplus \fr{q}$.  The subspaces $\fr{p}$ and $\fr{q}$ can be naturally identified with the tangent spaces $T_o(K/H)$ and $T_{\pi(o)}(G/K)$ respectively, where $\pi:G/H\rightarrow G/K$ is the natural projection. The following result of Alekseevsky and Nikonorov gives sufficient conditions for the induced inner products, $\left.g( \ ,\ )_o\right|_{\fr{p}\times \fr{p}}$ and $\left.g( \ ,\ )_o\right|_{\fr{q}\times \fr{q}}$, to define g.o. metrics on the spaces $K/H$ and $G/K$ respectively.

\begin{prop}\label{AlNiInd}\emph{(\cite{AlNi})}
Let $(G/H,g)$ be a g.o. space with $G$ compact and with corresponding metric endomorphism $A$. Then any closed, connected subgroup $K$ of $G$, which contains $H$, has the totally geodesic orbit $P=Ko=K/H$ which is a g.o. space (with respect to the induced metric $\left.g( \ ,\ )_o\right|_{\fr{p}\times \fr{p}}$). Moreover, if the space $\fr{p}:=\fr{k}\cap \fr{m}$ is $A$-invariant then $[\fr{k},\fr{m}^{\bot}]\subseteq \fr{m}^{\bot}$, and the product $\left.g( \ ,\ )_o\right|_{\fr{q}\times \fr{q}}$ is $\operatorname{Ad}_K$-invariant and defines a $G$-invariant g.o. metric on the homogeneous manifold $G/K$. The projection $\pi: G/H \rightarrow G/K$ is a Riemannian submersion
with totally geodesic fibers such that the fibers and the base are g.o. spaces.
\end{prop}

\noindent As a result, we obtain the following.

\begin{corol}\label{AlNiInd1}
Let $(G/H,g)$ be a g.o. space with corresponding metric endomorphism $A$, and let $K$ be a closed, connected subgroup of $G$ that contains $H$. If the space $\fr{p}:=\fr{k}\cap \fr{m}$ is $A$-invariant then the maps $\left.A\right|_{\fr{p}}$ and $\left.A\right|_{\fr{q}}$ are endomorphisms of $\fr{p}$ and $\fr{q}$ respectively.  Moreover, $\left.A\right|_{\fr{p}}$ defines a $K$-invariant g.o. metric on $K/H$ and $\left.A\right|_{\fr{q}}$ defines a $G$-invariant g.o. metric on $G/K$.
\end{corol}

\begin{proof}
If the space $\fr{p}$ is $A$ invariant then $\left.A\right|_{\fr{p}}$ is an endomorphism.  Moreover, the symmetry of $A$ with respect to $B$ implies that $B(A\fr{q},\fr{p})\subseteq  B(\fr{q},A\fr{p})\subseteq B(\fr{q},\fr{p})=\{0\}$.  Therefore, $A\fr{q}$ is orthogonal to $\fr{p}$, and the decomposition $\fr{m}=\fr{p}\oplus \fr{q}$ yields $A\fr{q}\subseteq \fr{q}$. Hence, $\left.A\right|_{\fr{q}}$ is an endomorphism.  Moreover, $g(X ,Y)_o=B(\left.A\right|_{\fr{p}}X,Y)$, for all $X,Y\in \fr{p}$, and  $g(X ,Y)_o=B(\left.A\right|_{\fr{q}}X,Y)$, for all $X,Y\in \fr{q}$.  On the other hand, by Proposition \ref{AlNiInd}, the products $\left.g( \ ,\ )_o\right|_{\fr{p}\times \fr{p}}$ and $\left.g( \ ,\ )_o\right|_{\fr{q}\times \fr{q}}$ define invariant g.o. metrics on the spaces $K/H$ and $G/K$ respectively, and hence the same is true for their metric endomorphisms $\left.A\right|_{\fr{p}}$ and $\left.A\right|_{\fr{q}}$.\end{proof}

\section{Reduction results for g.o. metrics on $G/S$}\label{Reduction}

We will henceforth assume the following.

\begin{assumption}\label{AsFla} The group $G$ is compact, connected, semisimple and $S$ is a torus in $G$.
\end{assumption}

Let $\fr{g}, \fr{s}$ be the Lie algebras of $G,S$ respectively, let $B$ denote the negative of the Killing form of $\fr{g}$ and fix a $B$-orthogonal decomposition $\fr{g}=\fr{s}\oplus \fr{m}$. Moreover, let 

\begin{equation*}C_G(S)=\{x\in G: xsx^{-1}=s \ \makebox{for all} \ s\in S\}\end{equation*}

\noindent be the centralizer of $S$ in $G$.  In this section we prove that any g.o. metric on $G/S$ induces a $G$-invariant g.o. metric on the space $G/C_G(S)$ and a bi-invariant metric on the compact Lie group $C_G(S)/S$ ($S$ is a normal subgroup of $C_G(S)$).  The space $G/C_G(S)$ is a generalized flag manifold; One of the equivalent definitions  of generalized flag manifolds is the following (see for example \cite{Ar1}).  

\begin{definition}\label{flagdef} Let $G$ be a compact and connected semisimple Lie group. A generalized flag manifold is a homogeneous space of the form $G/C_G(S)$, where $S$ is a torus in $G$. If $S$ is a maximal torus then $C_G(S)=S$, and $G/S$ is called a full flag manifold.   
\end{definition} 

The $G$-invariant Riemannian g.o. metrics on flag manifolds were classified by Alekseevski and Arvanitoyeorgos in \cite{AlAr}, who reduced the classification to flag manifolds of simple Lie groups. We state the classification result which will be used in the sequel.

\begin{theorem}\label{AlArv}\emph{(\cite{AlAr})}
The only flag manifolds  $M=G/K$ of a simple Lie group $G$ which admit an invariant g.o. metric, not homothetic to the standard metric, are the manifolds $SO(2l+l)/U(l)$ of complex structures in $R^{2l+2}$ and the complex projective space $Sp(l)/(U(1)\times Sp(l-1))$. These manifolds admit a one-parameter family of invariant metrics (up to a scaling). All these metrics are g.o. metrics and are weakly symmetric.  Moreover, the corresponding spaces are not naturally reductive.
\end{theorem}

Before we proceed to the main results of this section, set $K:=C_G(S)$. Since $S$ is connected, the Lie algebra of $K$ is 
 
 \begin{equation}\label{LieAlgC}\fr{k}=\{Y\in \fr{g}:[Y,X]=0 \ \ \makebox{for all} \ \ X\in \fr{s}\}.\end{equation}
 
 \noindent   Given that $\fr{s}\subseteq \fr{k}$, we consider the $B$-orthogonal decompositions $\fr{k}=\fr{s}\oplus \fr{p}$ and $\fr{g}=\fr{k}\oplus \fr{q}$, so that the tangent space at the origin of $G/S$ decomposes as $\fr{m}=\fr{p}\oplus \fr{q}$.  Here $\fr{p}$ is the Lie algebra of $C_G(S)/S$ and $\fr{q}$ is the tangent space at the origin of the flag manifold $G/C_G(S)$.  We have the following.

\begin{prop}\label{Ind}
Let $(G/S, A)$ be a g.o. space, where $G$ is a compact, connected and semisimple Lie group and $S$ is a torus. Then both maps $\left.A\right|_{\fr{p}}$ and $\left.A\right|_{\fr{q}}$ are endomorphisms, and the spaces $\big(C_G(S)/S, \left.A\right|_{\fr{p}}\big)$ and $\big(G/C_G(S), \left.A\right|_{\fr{q}}\big)$ are g.o. spaces.  Moreover, the metric $\left.A\right|_{\fr{p}}$ on the Lie group $C_G(S)/S$ is bi-invariant.
\end{prop} 

\begin{proof} Firstly, the group $K:=C_G(S)$ is connected (\cite{Hel} p. 287).  Then the first part of Proposition \ref{Ind} will follow from Corollary \ref{AlNiInd1} if we show that $\fr{p}$ is $A$-invariant.  To this end, relation \eqref{LieAlgC} and the definition of $\fr{p}$ imply that $\fr{p}=\fr{k}\cap \fr{m}=\{Y\in \fr{m}:[Y,X]=0 \ \ \makebox{for all} \ \ X\in \fr{s}\}$. Since $A$ defines a $G$-invariant metric on $G/S$, $A$ is $\operatorname{ad}_{\fr{s}}$-equivariant. Therefore, $[AY,X]=A[Y,X]=0$ for all $Y\in \fr{p}$ and $X\in \fr{s}$.  Hence, $AY\in\ \fr{p}$ for all $Y\in \fr{p}$, which yields the first part of the proposition.

To prove the bi-invariance of $\left.A\right|_{\fr{p}}$, firstly observe that since $S$ is compact, $K$ is also compact, and thus $K/S$ is a compact Lie group.  Moreover, $K/S$ is connected as $K$ is connected.  The metric endomorphism $\left.A\right|_{\fr{p}}$ corresponds to an inner product $\left.B(\left.A\right|_{\fr{p}}\cdot \ , \cdot \ )\right|_{\fr{p}\times \fr{p}}$ on the Lie algebra $\fr{p}$ of $K/S$, and hence it defines a left-invariant metric on $K/S$.  By Lemma \ref{Bi}, it suffices to show that $\left.A\right|_{\fr{p}}$ satisfies the equation $[AX,X]=0$ for all $X\in \fr{p}$.  Since $(G/S,A)$ is a g.o. space, by Proposition \ref{GOCond} there exists a map $\xi:\fr{m}\rightarrow \fr{s}$ such that 
\begin{equation}\label{Tgo}[\xi(X)+X,AX]=0 \ \ \makebox{for all} \ \ X\in \fr{p}.\end{equation}

\noindent But $A\fr{p}\subseteq \fr{p}\subset \fr{k}$.  Hence, $[\xi(X),AX]=0$ which, in view of Equation \eqref{Tgo}, yields $[AX,X]=0$. \end{proof}

 We have proven that any $G$-invariant g.o. metric $A$ on $G/S$ has the form $A=\begin{pmatrix} \left.A\right|_{\fr{p}} & 0\\
0 & \left.A\right|_{\fr{q}}\end{pmatrix}$, where each block defines a g.o. metric on the corresponding spaces $G/K$ and $K/S$, where $K=C_G(S)$. Furthermore, consider the decomposition

\begin{equation*}\fr{g}=\fr{g}_1\oplus \cdots \oplus \fr{g}_k,\end{equation*}

\noindent of $\fr{g}$ into the $B$-orthogonal, Lie-algebra direct sum of its simple ideals $\fr{g}_j$, $j=1,\dots,k$.  Since the algebra $\fr{k}$ has maximal rank in $\fr{g}$, we have $\fr{k}=\fr{k}_1\oplus \cdots \oplus \fr{k}_k$, where $\fr{k}_j=\fr{k}\cap\fr{g}_j$.  Consider the $B$-orthogonal decompositions $\fr{g}_j=\fr{k}_j\oplus \fr{q}_j, \ \ j=1,\dots,k$.  Then the tangent space $\fr{q}$ at the origin of the flag manifold $G/K$ decomposes as $\fr{q}=\fr{q}_1\oplus \cdots \oplus \fr{q}_k$, where $[\fr{q}_i,\fr{q}_j]=\{0\}$ for $i\neq j$, and each space $\fr{q}_j$ is the tangent space of a flag manifold $G_j/K_j$ with $G_j$ simple. The restriction $\left.A\right|_{\fr{q}}$ defines a $G$-invariant g.o. metric on $G/K$, therefore $\left.A\right|_{\fr{q}}$ is $\operatorname{Ad}_{K}$-equivariant, and hence $\operatorname{ad}_{\fr{k}_j}$-equivariant.  Since $K_j$ is connected ($K_j$ is the centralizer of a torus in $G_j$), each restriction $\left.A\right|_{\fr{q}_j}$ is $\operatorname{Ad}_{K_j}$-equivariant and thus $\left.A\right|_{\fr{q}_j}$ is an endomorphism of $\fr{q}_j$, which in turn defines a g.o. metric on the corresponding flag manifold $G_j/K_j$.  In summary, we have the following.

\begin{prop}\label{Red1}
Let $G/S$ be homogeneous space where $G$ is a compact, connected and semisimple Lie group and $S$ is a torus.  Let $K$ be the centralizer of $S$ in $G$, and let $\fr{g},\fr{s},\fr{k}$ be the Lie algebras of $G,S,K$ respectively.  Denote by $B$ the negative of the Killing form of $\fr{g}$ and consider the $B$-orthogonal decomposition $\fr{g}=\fr{s}\oplus \fr{m}$.  Moreover, decompose $\fr{g}=\bigoplus_{j=1}^k\fr{g}_j$ into the sum of its simple ideals $\fr{g}_j$, set $\fr{k}_j:=\fr{k}\cap \fr{g}_j$ and let $\fr{q}_j$ be the $B$-orthogonal complement of $\fr{k}_j$ in $\fr{g}_j$.  Then $\fr{m}$ decomposes as  

 \begin{equation*}\label{Odec1}\fr{m}=\fr{p}\oplus \fr{q}_1\oplus \cdots \oplus \fr{q}_k,\end{equation*}  
 
\noindent where $\fr{p}$ is the $B$-orthogonal complement of $\fr{s}$ in $\fr{k}$ and coincides with the Lie algebra of the compact Lie group $K/S$, while each $\fr{q}_j$ is the $B$-orthogonal complement of $\fr{k}_j$ in $\fr{g}_j$ and coincides with the tangent space at the origin of a generalized flag manifold $G_j/K_j$ with $G_j$ simple. 
  Moreover, any $G$-invariant g.o. metric $A$ on $G/S$ has the form 
  
  \begin{equation*}A=\begin{pmatrix} 
 \left.A\right|_{\fr{p}} & 0 & \cdots &0\\
  0 & \left.A\right|_{\fr{q}_1} & \cdots &0\\
  0&0 &\ddots &0\\
  0& 0& \cdots & \left.A\right|_{\fr{q}_k}
  \end{pmatrix},
\end{equation*}

\noindent where the endomorphism $\left.A\right|_{\fr{p}}$ defines a bi-invariant (g.o.) metric on $K/S$ and each endomorphism $\left.A\right|_{\fr{q}_j}$ defines a $G_j$-invariant g.o. metric on the corresponding flag manifold $G_j/K_j$. 
\end{prop}

\section{Lie theoretic description of flag manifolds and their isotropy algebras}\label{G/K}

In the previous section we showed that an invariant g.o. metric $A$ on $G/S$ induces invariant g.o. metrics on flag manifolds of simple Lie groups.  In this section we will provide a Lie theoretic description of the structure of those flag manifolds and the Lie algebras of their isotropy groups.  The section is concluded with two lemmas which will be useful for further simplifying the g.o. metrics on $G/S$.  For further details on the preliminary results of this section and on the Lie theoretic description of flag manifolds, we refer to \cite{Al}, \cite{AlAr}, \cite{ArChSa}, \cite{BoFoRo}.  To be consistent with the notation of the previous section, assume that $G_j/K_j$ is a generalized flag manifold where $G_j$ is a compact, connected, simple Lie group and $K_j$ is the centralizer of a torus in $G_j$.  Let $\fr{g}_j,\fr{k}_j$ be the Lie algebras of the groups $G_j,K_j$ respectively.

\subsection{Some preliminaries on root systems}\label{root}

  We firstly recall some useful preliminary results on root systems for which we refer to \cite{Hel} and \cite{Hum}. 
  
  Denote by $\fr{g}_j^{\mathbb C}$ the complexification of the algebra $\fr{g}_j$.  Let $R_j\subset (\fr{h}_j^{\mathbb C})^*$ be the \emph{root system} of $\fr{g}_j^{\mathbb C}$ with respect to a Cartan subalgebra $\fr{h}_j^{\mathbb C}$ of $\fr{g}_j^{\mathbb C}$. Let $R_j^+$ be the subset of \emph{positive roots} and let $\Pi_j=\{\alpha_1,\alpha_2,\dots,\alpha_{l_j}\}\subset R_j^+$ be the set of \emph{simple roots} so that each $\alpha\in R_j$ is written as $\alpha=\sum_{i=1}^{l_j}{c_i\alpha_i}$, where either all $c_i$ are non-negative integers (in which case $\alpha\in R_j^+$) or all $c_i$ are non-positive integers.  Consider the root decomposition $\fr{g}_j^{\mathbb C}=\fr{h}_j^{\mathbb C}\oplus \sum_{\alpha\in R_j}{\fr{g}^{\alpha}}$, where 
  
  \begin{equation*}\fr{g}^{\alpha}=\{X\in \fr{g}_j^{\mathbb C}:[H,X]=\alpha(H)X \ \makebox{for all} \ H\in \fr{h}_j^{\mathbb C}\}\end{equation*}
  
  \noindent  are the corresponding (one-dimensional) \emph{root spaces}.  Recall also the relation

\begin{equation}\label{RootBr}
[\fr{g}^{\alpha},\fr{g}^{\beta}]=\left\{ 
\begin{array}{lll}
\fr{g}^{\alpha+\beta},  \quad \makebox{if $\alpha+\beta\in R_j$}, \ \\

\{0\},   \quad \makebox{if $\alpha+\beta\notin R_j$}.
\end{array}
\right.
\end{equation}

We denote by $B_j$ the Killing form of $\fr{g}_j^{\mathbb C}$.  The restriction of $B_j$ to $\fr{g}_j$ is the Killing form of $\fr{g}_j$ (\cite{Hel} p. 180), and we will denote it again by $B_j$. 
 The form $B_j$ is non-degenerate on $\fr{h}_j^{\mathbb C}$, and we define an isomorphism $(\fr{h}_j^{\mathbb C})^*\rightarrow \fr{h}_j^{\mathbb C}$ by assigning to each root $\alpha \in (\fr{h}_j^{\mathbb C})^*$ a unique \emph{root vector} $H_{\alpha}\in \fr{h}_j^{\mathbb C}$ via the equation $B_j(H_{\alpha},H)=\alpha(H) \ \makebox{for all} \ H\in \fr{h}_j^{\mathbb C}$.  Moreover, we have $\alpha(H_{\alpha})\neq 0$ for all $\alpha \in R_j$.  For $\alpha,\beta\in R_j$, we set 

\begin{equation*}\label{prodc}(\alpha,\beta):=B_j(H_{\alpha},H_{\beta}) \ \makebox{and} \  \langle \alpha,\beta \rangle:=2\frac{(\alpha,\beta)}{(\beta,\beta)}.\end{equation*}

\noindent If $\alpha,\beta \in R_j$ with $\alpha \neq \pm \beta$, then there exist non-negative integers $r,q$ such that $\alpha+m\beta\in R_j$ for all integers $m$ satisfying $-r\leq m\leq q$.  Moreover, assume that $r$ is the maximum non-negative integer such that $\alpha-r\beta\in R_j$ and $q$ is the maximum non-negative integer such that $\alpha+q\beta\in R_j$.  Then

\begin{equation*}r-q=\langle \alpha ,\beta\rangle=2\frac{(\alpha,\beta)}{(\beta,\beta)}.\end{equation*}

\noindent  In turn, we have the following.

\begin{lemma}\label{Hum}
Let $\alpha,\beta\in R_j$ with $\alpha \neq \pm \beta$.  Then the following are true:\\
\noindent \emph{(i)} If $\alpha+\beta\notin R_j$ and $\alpha-\beta \notin R_j$ then $(\alpha,\beta)=0$.\\
\noindent\emph{(ii)} If $\alpha+\beta \in R_j$ and $\alpha-\beta \notin R_j$ then $(\alpha,\beta)<0$.\\
\noindent \emph{(iii)} If $(\alpha,\beta)<0$ then $\alpha+\beta\in R_j$.\\
\noindent \emph{(iv)} If $\alpha,\beta$ are simple then $(\alpha,\beta)\leq 0$. 
 \end{lemma}

Finally, we note that since $\fr{g}_j$ is simple and since it can be regarded as the Lie algebra of a compact matrix group, the complexification $\fr{g}_j^{\mathbb C}$ is also simple (\cite{Ha} p. 185). Therefore, the root system $R_j$ is \emph{irreducible} or equivalently, if there exist sets $\Pi_{j_1},\Pi_{j_2}$ such that $\Pi_j=\Pi_{j_1}\cup \Pi_{j_2}$ with $\Pi_{j_1}\cap\Pi_{j_2}=\emptyset$ and $(\Pi_{j_1},\Pi_{j_2})=\{0\}$, then either $\Pi_{j_1}$ or $\Pi_{j_2}$ is empty.

\subsection{Lie theoretic description of flag manifolds}\label{GFM}

We will now focus on the flag manifolds $G_j/K_j$, of the compact simple Lie groups $G_j$, their isotropy algebras $\fr{k}_j$ and the corresponding tangent spaces $\fr{q}_j$, i.e. the orthogonal complements of $\fr{k}_j$ in $\fr{g}_j$ with respect to the Killing form.  Denote by $\fr{k}_j^{\mathbb C}, \fr{q}_j^{\mathbb C}\subset \fr{g}_j^{\mathbb C}$ the complexifications of $\fr{k}_j,\fr{q}_j$ respectively.

  There exists a subset $\Pi_{K_j}$ of $\Pi_j$ such that $\fr{k}_j^{\mathbb C}=\fr{h}_j^{\mathbb C}\oplus \sum_{\alpha \in R_{K_j}}\fr{g}^{\alpha}$, where 
   
   \begin{equation*}R_{K_j}:=\operatorname{span}_{\mathbb Z}(\Pi_{K_j})=\{\alpha\in R_j:\alpha=\sum_{\alpha_i\in \Pi_{K_j}}{c_i\alpha_{i}}\}.\end{equation*}  We set $\Pi_{M_j}:=\Pi_j\setminus \Pi_{K_j}$ and we consider the set $R_{M_j}:=R_j\setminus R_{K_j}$ of the \emph{complementary roots}, so that $\Pi_j=\Pi_{K_j}\cup \Pi_{M_j}$ and $R_j=R_{K_j}\cup R_{M_j}$.  Then  $\fr{q}_j^{\mathbb C}=\sum_{\alpha\in R_{M_j}}{\fr{g}^{\alpha}}$.

\begin{remark}\label{rem1}The flag manifold $G_j/K_j$ is completely determined by the subset $\Pi_{M_j}\neq \emptyset $ of $\Pi_j$. The construction of $G_j/K_j$ from $\Pi_{M_j}$ is visualized in the \emph{painted Dynkin diagram} of $\fr{g}_j^{\mathbb C}$ (see for example \cite {AlAr}, \cite{BoFoRo} or Example \ref{EX2} in the next subsection), where the painted roots correspond to $\Pi_{M_j}$.
\end{remark}

  To obtain the root space decompositions for the real forms $\fr{g}_j,\fr{k}_j,\fr{q}_j$, we firstly make the identification 
 
\begin{equation*}\fr{h}_j=\operatorname{span}_{\mathbb R}\{\sqrt{-1}H_{\alpha}:\alpha\in R_j\}.\end{equation*}

\noindent Consider the sets $R_{K_j}^+=R_{K_j}\cap R_j^+$, $R_{M_j}^+=R_{M_j}\cap R_j^+$, and observe that $\Pi_{K_j}\subset R_{K_j}^+$ and $\Pi_{M_j}\subset R_{M_j}^+$.  The algebra $\fr{g}_j^{\mathbb C}$ admits a \emph{Weyl basis} $\{H_{\alpha},E_{\alpha}:\alpha\in R_j\}$, where $E_{\alpha}\in \fr{g}^{\alpha}$ and $[E_{\alpha},E_{\beta}]=N_{\alpha,\beta}E_{\alpha+\beta}$.  Here, $N_{\alpha,\beta}\neq 0$ if and only if $\alpha+\beta\in R_j$, and $N_{-\alpha,-\beta}=-N_{\alpha,\beta}$. Consider the vectors $A_{\alpha}:=E_{\alpha}-E_{-\alpha}$ and $B_{\alpha}:=\sqrt{-1}(E_{\alpha}+E_{-\alpha})$.  We set
  
  \begin{equation}\label{Ma}\fr{m}^{\alpha}:=\operatorname{span}_{\mathbb R}\{A_{\alpha},B_{\alpha}\}, \ \ \alpha\in R_j^+.\end{equation}  
     
\noindent The real forms $\fr{g}_j,\fr{k}_j,\fr{q}_j$ admit the orthogonal decompositions (with respect to the Killing form)

\begin{equation}\label{koin}\fr{g}_j=\fr{h}_j\oplus \sum_{\alpha\in R_j^+}{\fr{m}^{\alpha}}, \  \ \ \fr{k}_j=\fr{h}_j\oplus \sum_{\alpha\in R_{K_j}^+}{\fr{m}^{\alpha}}\  \ \ \makebox{and} \ \ \ \fr{q}_j=\sum_{\alpha\in R_{M_j}^+}{\fr{m}^{\alpha}}.\end{equation} 
   
 For $X_{\alpha}=cA_{\alpha}+dB_{\alpha} \in \fr{m}^{\alpha}$, set $\bar{X}_{\alpha}:=cB_{\alpha}-dA_{\alpha}\in \fr{m}^{\alpha}$.  Then for any $\sqrt{-1}H_{\beta}\in \fr{h}_j$ we obtain $[\sqrt{-1}H_{\beta},X_{\alpha}]=\alpha(H_{\beta})\bar{X}_{\alpha}$.  Therefore, $[\fr{h}_j,\fr{m}^{\alpha}]\subseteq \fr{m}^{\alpha}$ and in particular

\begin{equation}\label{Xa}
 [Y,X_{\alpha}]=\alpha(Y)\bar{X}_{\alpha}\ \ \makebox{for all}  \ \ Y\in \fr{h}_j, \ X_{\alpha}\in \fr{m}^{\alpha}.
\end{equation}    
 
\noindent Moreover, in view of relation \eqref{RootBr} and the definition of the spaces $\fr{m}^{\alpha}$ we obtain

\begin{equation}\label{ReRootBr}
[\fr{m}^\alpha,\fr{m}^\beta]\subseteq \left\{ \begin{array}{ll}  \fr{m}^{\alpha+\beta}+\fr{m}^{\left|\alpha-\beta \right|}, \quad \mbox{if} \quad \alpha+\beta \in R_j \quad \mbox{or} \quad \alpha-\beta \in R_j, \\
\left\{ {0} \right\}, \quad \mbox{otherwise}.
\end{array}
\right., 
\end{equation}

\noindent More importantly, if $X_{\alpha}\in \fr{m}^{\alpha}\setminus \{0\}$ and $X_{\beta}\in \fr{m}^{\beta}\setminus \{0\}$, such that $\alpha+\beta\in R_j$ or $\alpha-\beta \in R_j$, then $[X_{\alpha},X_{\beta}]\neq 0$. Observe also that if $\alpha\in R_{K_j}$, $\beta \in R_{M_j}$ and $\alpha+\beta \in R_j$, then $\alpha+\beta \in R_{M_j}$.  Hence, in view of decompositions \eqref{koin} and relations \eqref{Xa} and \eqref{ReRootBr}, we obtain $[\fr{k}_j,\fr{q}_j]\subseteq \fr{q}_j$.

 The center $\fr{t}_j$ of $\fr{k}_j$ lies in $\fr{h}_j$ and is given by 

\begin{equation}\label{center}\fr{t}_j=\{X\in \fr{h}_j:\alpha(X)=0 \ \makebox{for all} \ \alpha\in R_{K_j}\}=\{X\in \fr{h}_j:\alpha(X)=0 \ \makebox{for all} \ \alpha\in \Pi_{K_j}\}.\end{equation}    
   
\noindent   We also have the orthogonal decomposition $\fr{h}_j=\fr{t}_j\oplus \fr{t}_j^{\prime}$, with respect to the Killing form, where 

\begin{equation}\label{TPJCC}\fr{t}_j^{\prime}=\operatorname{span}_{\mathbb R}\{\sqrt{-1}H_{\alpha}:\alpha\in \Pi_{K_j}\}.\end{equation}

\noindent  Hence, the isotropy algebra $\fr{k}_j$ of the flag manifold $G_j/K_j$ decomposes as

\begin{equation}\label{IsotAlg}\fr{k}_j=\fr{t}_j\oplus \fr{t}_j^{\prime}\oplus \sum_{\alpha\in R_{K_j}^+}\fr{m}^{\alpha}.\end{equation}

\subsection{The flag manifolds $SO(2l+1)/U(l)$ and $Sp(l)/(U(1)\times Sp(l-1))$}\label{2isot}

Assume that $G_j/K_j$ is one of the flag manifolds $SO(2l+1)/U(l)$ or $Sp(l)/(U(1)\times Sp(l-1))$ which, by Theorem \ref{AlArv}, are the only flag manifolds of simple Lie groups admitting non-standard g.o. metrics.  We will briefly describe the tangent space of $G_j/K_j$ and the $G_j$-invariant g.o. metrics.  We will also provide a simple example for $SO(7)/U(3)$.  To those ends, we take into account results from \cite{AlAr} and \cite{ArCh}.  

  For both manifolds $G_j/K_j$, we have the decomposition $\fr{q}_j=\fr{q}_{j}^1\oplus \fr{q}_j^2$ of the tangent space at the origin into two $\operatorname{Ad}_{K_j}$-invariant, irreducible and inequivalent submodules $\fr{q}^1_j$ and $\fr{q}^2_j$.  The metric endomorphism $A:\fr{q}_j\rightarrow \fr{q}_j$ of any $G_j$-invariant metric on $G_j/K_j$ is given (up to homothety) by the diagonal form $\begin{pmatrix}
\left.\operatorname{Id}\right|_{\fr{q}_j^1} && 0 \\
 0 && \left.\lambda_j\operatorname{Id}\right|_{\fr{q}_j^2}\end{pmatrix}, \quad \lambda_j>0$.  According to the results in \cite{AlAr}, and more specifically Theorem \ref{AlArv}, all the above metrics on $G_j/K_j$ are g.o. metrics. 

To further describe the submodules $\fr{q}_j^1,\fr{q}_j^2$, recall the subsets $\Pi_{K_j},\Pi_{M_j},R_{K_j},R_{M_j}$ of the root system $R_j$ of $\fr{g}_j^{\mathbb C}$, as given in subsection \ref{GFM}.  For a root $\alpha \in R_j$ and a simple root $\alpha_i\in \Pi_j$, let $c_i(\alpha)$ be the (integer) coefficient of $\alpha$ in $\alpha_i$.  For each manifold $G_j/K_j$, we have $\Pi_{M_j}=\{\alpha_0\}$ where $\alpha_0$ is a simple root such that $\operatorname{max}\{c_0(\alpha):\alpha \in R_j\}=2$.  Then $R_{M_j}=\{\alpha\in R:c_0(\alpha)\neq 0\}$ and $R_{M_j}^+=R^1_{M_j}\cup R_{M_j}^2$, where

\begin{equation*}R_{M_j}^1=\{\alpha \in R_{M_j}^+:c_0(\alpha)=1\}\ \makebox{and} \ R_{M_j}^2=\{\alpha \in R_{M_j}^+:c_0(\alpha)=2\}.\end{equation*}

\noindent Moreover, we have 

\begin{equation} \label{Qj}
\fr{q}_j^1=\sum_{\alpha\in R_{M_j}^1}{\fr{m}^{\alpha}} \ \ \makebox{and} \ \ \fr{q}_j^2=\sum_{\alpha\in R_{M_j}^2}{\fr{m}^{\alpha}},
\end{equation}

\noindent where the subspaces $\fr{m}^{\alpha}$ are given by relation \eqref{Ma}. Observe that if $\beta_1\in R_{M_j}^1$, $\beta_2\in R_{M_j}^2$ and $\beta_1+\beta_2\in R_j$ or $\beta_1-\beta_2\in R_j$, then $\beta_1+\beta_2\in R_{M_j}^1$ or $\beta_1-\beta_2\in R_{M_j}^1$ respectively.  Therefore, relations \eqref{ReRootBr} yield

\begin{equation}\label{BRKT}[\fr{q}_j^1,\fr{q}_j^2]\subseteq \fr{q}_j^1.\end{equation}

\begin{example}\label{EX2}\emph{(The flag manifold $G_j/K_j=SO(7)/U(3)$)}  The root system $R_j=\fr{b}_3$ of $\fr{g}_j^{\mathbb C}=\fr{so}(7)^{\mathbb C}$ is given by $\fr{b}_3=\{\pm e_l \ (1\leq l\leq 3), \ \pm e_l\pm e_m \ (1\leq l\neq m\leq 3, \ \pm \ \makebox{independent}\}$, corresponding to the space $V=\operatorname{span}_{\mathbb R}\{e_1,e_2,e_3\}$ (see for example \cite{Hel} p. 462).  The set of simple roots is $\Pi_j=\{\alpha_1,\alpha_2,\alpha_3\}$ where $\alpha_1=e_1-e_{2}$, $\alpha_2=e_2-e_{3}$ and $\alpha_3=e_3$.  The flag manifold $SO(7)/U(3)$ corresponds to $\Pi_{M_j}=\{\alpha_3\}$, painted black in the Dynkin diagram of $\fr{b}_3$ as follows.

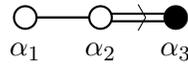
\begin{center}
  \begin{tikzpicture}[scale=0.5]
  \foreach \x in {1,2}
    \draw[xshift=\x cm,thick] (\x cm,0) circle (.3cm) node[below=2mm]{$\alpha_{\x}$};
   \draw[xshift=3 cm,thick,fill=black] (3 cm, 0) circle (.3 cm)  node[below=2mm]{$\alpha_3$};
   \foreach \y in {1.15,1.15}
    \draw[xshift=\y cm,thick] (\y cm,0) -- +(1.4 cm,0);
    \draw[thick] (4.3 cm, .1 cm) -- +(1.4 cm,0);
    \draw[thick] (4.3 cm, -.1 cm) -- +(1.4 cm,0);
     \draw(5.2,0) -- + (120:0.4) ; (5.2,0) -- + (-120:0.4);
     \draw(5.2,0) -- + (240:0.4) ; (5.2,0) -- + (-240:0.4);
  \end{tikzpicture}
\end{center}

\noindent Then $R^+_{K_j}=\{\alpha_1,\alpha_2,\alpha_1+\alpha_2\}$ and $R^+_{M_j}=\{\alpha_3,\alpha_2+\alpha_3,\alpha_2+2\alpha_3,\alpha_1+\alpha_2+\alpha_3,\alpha_1+\alpha_2+2\alpha_3,\alpha_1+2\alpha_2+2\alpha_3\}$.  Moreover,  $R^1_{M_j}=\{\alpha_3,\alpha_2+\alpha_3,\alpha_1+\alpha_2+\alpha_3\}$ and $R^2_{M_j}=\{\alpha_2+2\alpha_3,\alpha_1+\alpha_2+2\alpha_3,\alpha_1+2\alpha_2+2\alpha_3\}$.  Therefore, 

\begin{eqnarray*}\fr{k}_j&=&\fr{u}(3)=\fr{h}_j\oplus (\fr{m}^{\alpha_1}+\fr{m}^{\alpha_2}+\fr{m}^{\alpha_1+\alpha_2}),\\
\fr{q}_j^1&=&\fr{m}^{\alpha_3}+\fr{m}^{\alpha_2+\alpha_3}+\fr{m}^{\alpha_1+\alpha_2+\alpha_3} \ \makebox{and}\\
\fr{q}_j^2&=&\fr{m}^{\alpha_2+2\alpha_3}+\fr{m}^{\alpha_1+\alpha_2+2\alpha_3}+\fr{m}^{\alpha_1+2\alpha_2+2\alpha_3}.
\end{eqnarray*}

\end{example}

\subsection{Two lemmas for flag manifolds.}\label{tech}

We conclude this section with two technical lemmas. The first lemma concerns the flag manifolds $SO(2l+1)/U(l)$ and $Sp(l)/(U(1)\times Sp(l-1))$. The second lemma establishes a general property of roots associated to flag manifolds. These lemmas will be used for simplifying the g.o. metrics on $G/S$ in Section \ref{mainresults}.

\begin{lemma}\label{F1lem}
Let $G_j/K_j$ be one of the flag manifolds $SO(2l+1)/U(l)$ or $Sp(l)/(U(1)\times Sp(l-1))$, and let $\fr{q}_j^1,\fr{q}_j^2$ be given by relations \eqref{Qj}.  Then the following are true:\\

\noindent \emph{(i)} $[\fr{q}_j^1,\fr{q}_j^2]\neq \{0\}$.\\
\noindent \emph{(ii)}  For any non-zero vector $\xi$ in the center $\fr{t}_j$ of $\fr{k}_j$ and for any non-zero vector $X_2$ in $\fr{q}_j^2$, $[\xi,X_2]$ is a non-zero vector in $\fr{q}_j^2$.
\end{lemma}

\begin{proof}
Each positive root can be written in the form $\alpha_{i_1}+\cdots +\alpha_{i_{n}}$ where $\alpha_{i_m}$, $m=1,\dots,n$, are simple, not necessarily distinct roots such that each partial sum $\alpha_{i_1}+\cdots +\alpha_{i_s}$, $s\leq n$, is a root (\cite{Hum} p. 50).  For part (i), recall that $\Pi_{M_j}=\{\alpha_0\}$ and recall the definitions of $R_{M_j}^1,R_{M_j}^2$ in subsection \ref{2isot}.  Choose a root $\widetilde{\alpha} \in R_{M_j}^2$.  By the definition of $R_{M_j}^2$, $\widetilde{\alpha}$ has the form 
 
 \begin{equation}\label{2a0}\widetilde{\alpha}=\sum_{\alpha_i\in \Pi_{K_j}}{c_i\alpha_i}+2\alpha_0.\end{equation}
 
 \noindent  Write $\widetilde{\alpha}$ in the form $\alpha_{i_1}+\cdots +\alpha_{i_{n}}$ such that each partial sum $\alpha_{i_1}+\cdots +\alpha_{i_{s}}$, $s\leq n$, is a root, and let $n_0$ be the second index among $1,\dots,n$ such that $\alpha_{i_{n_0}}=\alpha_0$.  Set
 
 \begin{equation*}\beta:=\alpha_{i_1}+\cdots +\alpha_{i_{n_0-1}} \ \ \makebox{and} \ \ \gamma:=\alpha_{i_1}+\cdots +\alpha_{i_{n_0}}.\end{equation*}
 
 \noindent The definition of $n_0$ implies that $\beta \in R_{M_j}^1$, $\gamma \in R_{M_j}^2$ and $\gamma-\alpha_0=\beta$.  Let $X_{\gamma}\in \fr{m}^{\gamma}\setminus \{0\}\subset \fr{q}_j^2\setminus \{0\}$ and $X_{\alpha_0}\in \fr{m}^{\alpha_0}\setminus \{0\}\subset \fr{q}_j^1\setminus \{0\}$. Then by relation \eqref{ReRootBr} and the short discussion following it, we conclude that $[X_{\gamma},X_{\alpha_0}]\neq 0$ which yields the desired result.\\

For part (ii), let $X_2\in \fr{q}_j^2\setminus \{0\}$, and let $\xi\in \fr{t}_j$ such that $[\xi,X_2]=0$.  We will prove that $\xi=0$.  Write $X_2=\sum_{\alpha\in R_{M_j}^2}{X_{\alpha}}$, where $X_{\alpha}\in \fr{m}^{\alpha}$, and choose a root $\widetilde{\alpha}\in R_{M_j}^2$ such that $X_{\widetilde{\alpha}}\neq 0$ (such a root exists since $X_2\neq 0$).  Our hypothesis that $[\xi,X_2]=0$, the fact that $\xi\in \fr{h}_j$, Equation \eqref{Xa} and the linear independence of the vectors $X_{\alpha}$ in the expression $X_2=\sum_{\alpha\in R_{M_j}^2}{X_{\alpha}}$ imply that 

\begin{equation}\label{inview}\widetilde{\alpha}(\xi)=0.\end{equation}

By the definition of $R_{M_j}^2$, $\widetilde{\alpha}$ has the form \eqref{2a0}.  On the other hand, the definition \eqref{center} of $\fr{t}_j$ implies that $\alpha_i(\xi)=0$ for all $\alpha_i\in \Pi_{K_j}$, and in particular $(\sum_{\alpha_i\in \Pi_{K_j}}{c_i\alpha_i})(\xi)=0$.  Hence, Equation \eqref{inview} yields $\alpha_0(\xi)=0$.  Since $\Pi_{M_j}=\{\alpha_0\}$, we deduce that $\alpha(\xi)=0$ for all $\alpha\in R_j=\operatorname{span}_{\mathbb Z}(\Pi_{K_j}\cup\Pi_{M_j})$.  Then in view of decomposition \eqref{koin} for $\fr{g}_j$ we conclude that $\xi$ lies in the center of $\fr{g}_j$.  But $\fr{g}_j$ is simple and hence $\xi=0$.\end{proof}

Before we proceed to the second lemma, we recall the sets $R_{K_j},R_{M_j}\subset R_j$ and $\Pi_{K_j},\Pi_{M_j}\subset \Pi_j$, corresponding to a flag manifold $G_j/K_j$, as defined in subsection \ref{GFM}. Recall also that $R_j$ is irreducible and that $\Pi_{M_j}\neq \emptyset$ and $R_{M_j}\neq \emptyset$.

\begin{lemma}\label{SumExis}Assume that $\Pi_{K_j}\neq \emptyset$.  Then for any root $\alpha\in R_{K_j}^+$ there exists a root $\beta\in R_{M_j}$ such that $\alpha+\beta\in R_{M_j}$.
\end{lemma}

\begin{proof} Firstly, observe that if $\alpha\in R_{K_j}^+$ and $\beta \in \Pi_{M_j}$ then $\alpha-\beta\notin R_j$, otherwise, the difference $\alpha-\beta$ would have a positive coefficient in at least one of the simple roots in $\Pi_{K_j}$ and a negative coefficient in the simple root $\beta$ which is impossible.  Secondly, if $\alpha\in R_{K_j}^+$, $\beta \in R_{M_j}$ and $\alpha+\beta\in R_j$, then $\alpha+\beta\in R_{M_j}$. Therefore, under the hypotheses of the lemma, it suffices to show that $\alpha+\beta \in R_j$.

Assume that $\Pi_{K_j}=\{\alpha_1,\dots,\alpha_s\}$.  For $\alpha\in R_{K_j}^+$, write 

\begin{equation}\label{Alp}\alpha=\sum_{i=1}^s{c_i\alpha_i} \ \ \makebox{where} \ \ c_i\geq 0 \ \ \makebox{and} \ \ \alpha_i \in \Pi_{K_j}.\end{equation}

\noindent Let $n=n(\alpha)$ be the number of zero coefficients $c_i$ in Equation \eqref{Alp}.  Then $0\leq n(\alpha)\leq s-1$. The proof of the lemma will be concluded if we show that for each $n\in \mathbb N$ with $0\leq n \leq s-1$ and for each $\alpha\in R_{K_j}^+$ with $n(\alpha)=n$ there exists a root $\beta \in R_{M_j}$ such that $\alpha+\beta\in R_j$.  To this end, we will use induction on $n$.

For $n=0$, assume on the contrary that there exists a root $\alpha\in R_{K_j}^+$, with $n(\alpha)=0$, such that $\alpha+\beta\notin R_j$ for all roots $\beta\in R_{M_j}$.  Then $\alpha+\beta\notin R_j$ for all roots $\beta \in \Pi_{M_j}$.  By taking also into account that $\alpha-\beta\notin R_j$ for all $\beta \in \Pi_{M_j}$, part (i) of Lemma \ref{Hum} implies that $(\alpha,\beta)=0$ for all $\beta\in \Pi_{M_j}$.  Therefore, Equation \eqref{Alp} yields 

\begin{equation}\label{Orth}0=(\alpha,\beta)=\sum_{i=1}^{s}c_i(\alpha_{i},\beta) \ \ \makebox{for all} \ \ \beta \in \Pi_{M_j}.\end{equation} 

\noindent On the other hand, part (iv) of lemma \ref{Hum} implies that $(\alpha_{i},\beta)\leq 0$, and given that $c_i>0$ for all $i=1,\dots,s$, we conclude from Equation \eqref{Orth} that $(\alpha_{i},\beta)=0$ for all $\alpha_{i}\in \Pi_{K_j}$ and $\beta \in \Pi_{M_j}$.  But this is equivalent to $(\Pi_{K_j},\Pi_{M_j})=\{0\}$ which, in view of the facts that $\Pi_{K_j},\Pi_{M_j}\neq \emptyset$ and $\Pi_{K_j}\cup \Pi_{M_j}=\Pi_j$, contradicts the irreducibility of $R_j$.  This completes the induction step for $n=0$.

Assume that the induction hypothesis holds for all $\alpha\in R_{K_j}^+$ with $n(\alpha)=N$, and let $\alpha\in R_{K_j}^+$ with $n(\alpha)=N+1$.  Then by the definition of $n(\alpha)$ and in view of Equation \eqref{Alp}, we obtain 

\begin{equation}\label{IndForm}\alpha=\sum_{m=1}^{s-(N+1)}{c_{i_m}\alpha_{i_m}}, \ \ \makebox{where} \ \ c_{i_m}>0 \ \ \makebox{and} \ \  \alpha_{i_m}\in \Pi_{K_j}.\end{equation}

\noindent  Set 

\begin{equation*}\Pi_{K_j}^1:=\{\alpha_{i_1},\dots ,\alpha_{i_{s-(N+1)}}\}\subset \Pi_{K_j},\ \ \Pi_{K_j}^2:=\Pi_{K_j}\setminus \Pi_{K_j}^1\ \  \makebox{and} \ \ R_{K_j}^1:=\operatorname{span}_{\mathbb Z}(\Pi_{K_j}^1).
\end{equation*}

\noindent  Then $\alpha\in R_{K_j}^1$, $\Pi_{K_j}=\Pi_{K_j}^1\cup \Pi_{K_j}^2$ and $\Pi_j=\Pi_{K_j}^1\cup\Pi_{K_j}^2\cup \Pi_{M_j}$.  Moreover, the sets $\Pi_{K_j}^1,\Pi_{K_j}^2$ and $\Pi_{M_j}$ are pairwise disjoint.  To conclude the induction for $n(\alpha)=N+1$, once more assume the following.
 
 \begin{assumption}\label{Asum} For all roots $\beta\in R_{M_j}$, $\alpha+\beta\notin R_j$.
\end{assumption}
 
 \noindent  We will show that this assumption leads to a contradiction, which will conclude the induction and the proof.  If Assumption \ref{Asum} is true, then by using the same argument as in the step for $n=0$, we deduce that 

\begin{equation}\label{Con}(\Pi_{K_j}^1,\Pi_{M_j})=\{0\}.\end{equation}
 
\noindent Moreover, there exists a root $\alpha_{i_0}\in \Pi_{K_j}^2$ such that $\alpha+\alpha_{i_0}\in R_j$; Otherwise, in view of expression \eqref{IndForm}, the same argument as in the step for $n=0$ would imply that $(\Pi_{K_j}^1,\Pi_{K_j}^2)=\{0\}$ which, along with Equation \eqref{Con}, would yield $(\Pi_{K_j}^1,\Pi_{K_j}^2\cup \Pi_{M_j})=\{0\}$ and thus would contradict the irreducibility of $R_j=\operatorname{span}_{\mathbb Z}\big(\Pi_{K_j}^1\cup\Pi_{K_j}^2\cup \Pi_{M_j}\big)$.  Since $\alpha+\alpha_{i_0}\in R_j$, Equation \eqref{IndForm} yields 

\begin{equation*}\alpha+\alpha_{i_0}=\sum_{m=1}^{s-(N+1)}{c_{i_m}\alpha_{i_m}}+\alpha_{i_0} \  \in R_{K_j}^+, \ \ \makebox{where} \ \ c_{i_m}>0.\end{equation*}

\noindent Hence $n(\alpha+\alpha_{i_0})=N$.  By the induction hypothesis, there exists a root $\beta \in R_{M_j}$ such that
 
 \begin{equation}\label{S2}\alpha+\alpha_{i_0}+\beta\in R_j.\end{equation}

\noindent Let $L$ be the largest non-negative integer such that $\beta+L(\alpha+\alpha_{i_0})\in R_j$.  Relation \eqref{S2} implies that $L\geq 1$.  We also note that since $\beta \neq \pm(\alpha+\alpha_{i_0})$, we have $\beta+k(\alpha+\alpha_{i_0})\in R_{M_j}$ for all $k$ with $0\leq k\leq L$.  We set $\beta_L:=-(\beta+L(\alpha+\alpha_{i_0}))\in R_{M_j}$.  By the definition of $L$, $\beta_L$ satisfies both the following properties:
 
\begin{equation*}\makebox{a)} \ \alpha+(\alpha_{i_0}+\beta_L)\in R_{M_j}\ \ \makebox{and} \ \ \makebox{b)}\ \alpha+(\alpha_{i_0}-\beta_L)\notin R_{M_j}.\end{equation*}
 
\noindent Taking into account the above properties and the fact that $\alpha+\alpha_{i_0}\in R_j$, part (ii) of Lemma \ref{Hum} implies that $(\alpha+\alpha_{i_0},\beta_L)<0$.  But relation \eqref{Con} implies that $(\alpha,\beta_L)=0$ and thus $(\alpha_{i_0},\beta_L)<0$.  Part (iii) of Lemma \ref{Hum} then yields $\alpha_{i_0}+\beta_L\in R_j$, and in fact $\alpha_{i_0}+\beta_L\in R_{M_j}$.  Then by property a), Assumption \ref{Asum} is contradicted and the induction is concluded.\end{proof}

\section{Necessary form for g.o. metrics on $G/S$}\label{mainresults}

Let $(G/S,g)$ be a g.o. space with $G$ compact, connected, semisimple and $S$ a torus in $G$.   In this section we obtain a necessary form for $g$.  As in the previous sections, let $\fr{g}=\fr{g}_1\oplus \cdots \oplus \fr{g}_k$ be the Lie algebra of $G$, where $\fr{g}_j$ are its simple ideals.  Let $B$ denote the negative of the Killing form of $\fr{g}$, let $\fr{s}$ be the Lie algebra of the torus $S$ and let $\fr{m}=T_o(G/S)$ be the $B$-orthogonal complement of $\fr{s}$ in $\fr{g}$.  Moreover, let $A:\fr{m}\rightarrow \fr{m}$ be the corresponding metric endomorphism of $g$, satisfying $g(X,Y)_o=B(AX,Y)$, $X,Y\in \fr{m}$.  

We recall that $\fr{k}=\fr{k}_1\oplus \cdots \oplus \fr{k}_k$ is the Lie algebra of the centralizer $K=C_G(S)$ of $S$ in $G$, where $\fr{k}_j=\fr{k}\cap \fr{g}_j$.  We also recall the $B$-orthogonal complement $\fr{p}$ of $\fr{s}$ in $\fr{k}$, which coincides with the Lie algebra of $K/S$,  and the $B$-orthogonal complements $\fr{q}_j$ of $\fr{k}_j$ in $\fr{g}_j$, each of which coincides with the tangent space at the origin of a flag manifold $G_j/K_j$. 

Let $\fr{h}=\fr{h}_1\oplus \cdots \oplus \fr{h}_k$ be a Cartan subalgebra of $\fr{g}$, where $\fr{h}_j$ are Cartan subalgebras of $\fr{g}_j$.  We recall the center $\fr{t}_j$ of $\fr{k}_j$, given by relation \eqref{center}.  Set 

\begin{equation}\label{Ef}\fr{f}:=\fr{t}_1\oplus \cdots \oplus \fr{t}_k.\end{equation}

\noindent Finally, let $\pi_j:\fr{g}\rightarrow \fr{g}_j$, $j=1,\dots,k$, be the linear projections of $\fr{g}$ on $\fr{g}_j$, which are Lie algebra homomorphisms.  We have the following.

\begin{lemma}\label{LIC}
The projection $\pi_j(\fr{s})$ lies in the center $\fr{t}_j$ of $\fr{k}_j$, $j=1,\dots ,k$. 
\end{lemma}

\begin{proof}  Taking into account the facts that $\pi_j$ is a homomorphism, $\fr{s}$ is abelian and $[\fr{g}_i,\fr{g}_j]=\{0\}$ for $i\neq j$, we obtain $[\pi_j(\fr{s}),\fr{s}]=[\pi_j(\fr{s}),\pi_j(\fr{s})]=\pi_j([\fr{s},\fr{s}])=\{0\}$.  Hence, $\pi_j(\fr{s})$ centralizes $\fr{s}$ and thus $\pi_j(\fr{s})$ lies in $\fr{k}_j=\fr{k}\cap \fr{g}_j$.  On the other hand, since $\fr{k}$ centralizes $\fr{s}$ we obtain $[\pi_j(\fr{s}),\fr{k}_j]=\pi_j([\fr{s},\fr{k}])=\{0\}$, and thus concluding that $\pi_j(\fr{s})\subseteq \fr{t}_j$.\end{proof}

  From Lemma \ref{LIC} we deduce that $\pi_j(\fr{s})\subseteq \fr{t}_j\subseteq \fr{h}_j$ and $\fr{s}\subseteq \fr{f}\subset \fr{h}$. Consider the $B$-orthogonal decomposition

\begin{equation*}\fr{f}=\fr{s}\oplus \fr{s}^{\prime}.\end{equation*}

\noindent By the definition of $\fr{p}$ and the decomposition \eqref{IsotAlg} of $\fr{k}_j$, we obtain $\fr{p}=\fr{s}^{\prime}\oplus \bigoplus_{j=1}^k\big(\fr{t}_j^{\prime}\oplus \sum_{\alpha\in R_{K_j}^+}\fr{m}^{\alpha}\big)$, where $\fr{t}_j^{\prime}$ is given by relation \eqref{TPJCC}.  Set

\begin{equation}\label{pij}\fr{p}_j:=\fr{t}_j^{\prime}\oplus \sum_{\alpha\in R_{K_j}^+}\fr{m}^{\alpha}.\end{equation}

\noindent In other words, $\fr{p}_j$ is the $B$-orthogonal complement of $\fr{t}_j$ in $\fr{k}_j$.  Then $\fr{p}=\fr{s}^{\prime}\oplus\fr{p}_1\oplus \cdots \oplus \fr{p}_k$. Therefore, 

\begin{equation}\label{DecMM}\fr{m}=\underbrace{\fr{s}^{\prime}\oplus\fr{p}_1\oplus \cdots \oplus \fr{p}_k}_{\fr{p}}\oplus\underbrace{\fr{q}_1\oplus\cdots \oplus \fr{q}_k}_{\fr{q}}.\end{equation}

  The main result of this section is the following.

\begin{theorem}\label{NecForm}
Let $(G/S,A)$ be a Riemannian geodesic orbit space where $G$ is a compact, connected, semisimple Lie group and $S$ is a torus.  Then $A:\fr{m}\rightarrow \fr{m}$ has the form

\begin{equation}\label{NecForm1}A=\begin{pmatrix} 
 \left.A\right|_{\fr{s}^{\prime}} & 0 & \cdots &0\\
  0 & \left.\lambda_1\operatorname{Id}\right|_{\fr{p}_1\oplus\fr{q}_1} & \cdots &0\\
  0&0 &\ddots &0\\
  0& 0& \cdots & \left.\lambda_k\operatorname{Id}\right|_{\fr{p}_k\oplus\fr{q}_k}
  \end{pmatrix}, \ \lambda_j>0.
\end{equation}

\end{theorem}

\subsection{Proof of Theorem \ref{NecForm}} Our main guide in proving Theorem \ref{NecForm} is Proposition \ref{Red1}, from which we recall the necessary form $A=\begin{pmatrix} 
 \left.A\right|_{\fr{p}} & 0 & \cdots &0\\
  0 & \left.A\right|_{\fr{q}_1} & \cdots &0\\
  0&0 &\ddots &0\\
  0& 0& \cdots & \left.A\right|_{\fr{q}_k}
  \end{pmatrix}$. Taking into account decomposition \eqref{DecMM}, in order to prove Theorem \ref{NecForm}, it suffices to prove the following three results.

 \begin{prop}\label{SimpFlag}
If $A$ is a g.o. metric on $G/S$ then each endomorphism $\left.A\right|_{\fr{q}_j}:\fr{q}_j\rightarrow \fr{q}_j$ is equal to $\lambda_j\operatorname{Id}$ for some $\lambda_j>0$. 
\end{prop}

\begin{prop}\label{SimIsot}
If $A$ is a g.o. metric on $G/S$ then each endomorphism $\left.A\right|_{\fr{p}_j}:\fr{p}_j\rightarrow \fr{p}_j$ is equal to $\lambda_j\operatorname{Id}$, where $\lambda_j$ is the same as in Proposition \ref{SimpFlag}.
\end{prop}

\begin{prop}\label{Endo}
If $A$ is a g.o. metric on $G/S$ then $\left.A\right|_{\fr{s}^{\prime}}$ is an endomorphism of $\fr{s}^{\prime}$.
\end{prop}

\noindent Let us firstly prove a simple linear-algebraic result. 

\begin{lemma}\label{LinAlg}
Let $V$ be a vector space and let $A:V\rightarrow V$ be a diagonalizable endomorphism.  Assume that $W$ is a subspace of $V$ such that any eigenvector of $A$ with non-zero projection on $W$ has eigenvalue $\lambda$.  Then $\left.A\right|_{W}=\lambda \operatorname{Id}$.
\end{lemma}

\begin{proof} Let $\lambda_1,\dots,\lambda_n$ be the pairwise distinct eigenvalues of $A$ that are different from $\lambda$. For an eigenvalue $\lambda_i$, let $\fr{m}_{\lambda_i}\subseteq V$ be the corresponding eigenspace. Since $A$ is diagonalizable, we can write $V=\fr{m}_{\lambda}\oplus \fr{m}_{\lambda}^{\bot}$, where $\fr{m}_{\lambda}^{\bot}=\bigoplus_{i=1}^n\fr{m}_{\lambda_i}$. Moreover, write $V=W\oplus W^{\bot}$.  The hypothesis of the lemma implies that $\fr{m}_{\lambda_i}\subseteq W^{\bot}$ for all $i=1,\dots,n$. Therefore, $\fr{m}_{\lambda}^{\bot}\subseteq W^{\bot}$ hence $W\subseteq \fr{m}_{\lambda}$.\end{proof}

We proceed with the proofs of Propositions \ref{SimpFlag}, \ref{SimIsot} and \ref{Endo}.\\

\emph{Proof of Proposition \ref{SimpFlag}.} \ If $\fr{q}_j$ is not the tangent space of one of the flag manifolds $SO(2l+1)/U(l)$ and $Sp(l)/(U(1)\times Sp(l-1))$, then by Theorem \ref{AlArv}, $\left.A\right|_{\fr{q}_j}$ is standard.  Hence, $\left.A\right|_{\fr{q}_j}=\lambda_j\operatorname{Id}$. If $\fr{q}_j$ is the tangent space of $SO(2l+1)/U(l)$ or $Sp(l)/(U(1)\times Sp(l-1))$, then we recall from subsection \ref{2isot} that $\left.A\right|_{\fr{q}_j}$ is homothetic to

\begin{equation}\label{OMN}\begin{pmatrix}
\left.\operatorname{Id}\right|_{\fr{q}_j^1} && 0 \\
 0 && \left.\lambda_j\operatorname{Id}\right|_{\fr{q}_j^2}\end{pmatrix},\ \ \lambda_j>0.\end{equation}
 
 \noindent  It suffices to prove that $\lambda_j=1$.
 
 By part (i) of Lemma \ref{F1lem}, we may choose $X_1\in \fr{q}_j^1$ and $X_2\in \fr{q}_j^2$ such that $[X_1,X_2]\neq 0$.  More specifically, relation \eqref{BRKT} yields
 
 \begin{equation}\label{neq}[X_1,X_2]\in \fr{q}_j^1\setminus \{0\}.\end{equation}
 
  \noindent Since $A$ defines a g.o. metric on $G/S$, by Proposition \ref{GOCond}, for the vector $X:=X_1+X_2\in \fr{q}_j$ there exists a vector $\xi(X)$ in the Lie algebra $\fr{s}$ of $S$ such that $[\xi(X)+X,AX]=0$.  Applying the map $\pi_j$ and taking into account the fact that $X,AX\in \fr{q}_j$, we obtain $0=\pi_j\big([\xi(X)+X,AX]\big)=[\pi_j(\xi(X))+X,AX]$. In view of the form \eqref{OMN} of $\left.A\right|_{\fr{q}_j}$, the last condition is equivalent to 

\begin{equation}\label{ftem1}[\pi_j(\xi(X)),X_1]+\lambda_j[\pi_j(\xi(X)),X_2]+(\lambda_j-1)[X_1,X_2]=0.
\end{equation}

The $\operatorname{ad}_{\fr{s}}$-invariance of $\fr{q}_j^1$ and $\fr{q}_j^2$ along with the fact that $\xi(X)\in \fr{s}$, imply that $[\pi_j(\xi(X)),X_i]=\pi_j([\xi(X),X_i])\in \fr{q}_j^i$, $i=1,2$.  Then along with the fact that $\fr{q}_j^1,\fr{q}_j^2$ are $B$-orthogonal and relation \eqref{neq}, Equation \eqref{ftem1} yields 

\begin{equation}\label{Saul}\lambda_j[\pi_j(\xi(X)),X_2]=0.\end{equation}

\noindent  On the other hand, Lemma \ref{LIC} implies that $\pi_j(\xi(X)) \in \fr{t}_j$; Along with the facts that $\lambda_j>0$ and $X_2\neq 0$, and by virtue of part (ii) of Lemma \ref{F1lem}, Equation \eqref{Saul} yields $\pi_j(\xi(X))=0$.  Substituting $\pi_j(\xi(X))=0$ into Equation \eqref{ftem1} we deduce that $(\lambda_j-1)[X_1,X_2]=0$. Since $[X_1,X_2]\neq 0$, we conclude that $\lambda_j=1$.\qed\\

\emph{Proof of Proposition \ref{SimIsot}}.  Let $\lambda_j$ be as in Proposition \ref{SimpFlag}. Recall the set $R_{K_j}$ defined in subsection \ref{GFM}.  If $R_{K_j}=\emptyset$, by decomposition \eqref{IsotAlg}, the definition \eqref{TPJCC} of $\fr{t}_j^{\prime}$ and relation \eqref{pij}, we obtain $\fr{p}_j=\{0\}$ and hence Proposition \ref{SimIsot}) holds trivially.  Assume that $R_{K_j}\neq \emptyset$.  Since $\left.A\right|_{\fr{p}}$ is a diagonalizable endomorphism and $\fr{p}_j\subseteq \fr{p}$, there exists an eigenvector $v\in \fr{p}$ of $\left.A\right|_{\fr{p}}$ with non-zero projection $u_{\fr{p}_j}$ on $\fr{p}_j$. Let $v$ be such an eigenvector and let $\lambda$ be the corresponding eigenvalue of $v$.  In view of Lemma \ref{LinAlg} for $V=\fr{p}$ and $W=\fr{p}_j$, it suffices to show that $\lambda=\lambda_j$.

According to the decomposition $\fr{p}=\fr{s}^{\prime}\oplus \fr{p}_1\oplus \cdots \oplus \fr{p}_k$, write $v=v_{\fr{s}^{\prime}}+v_{\fr{p}_1}+\cdots +v_{\fr{p}_k}$ (here the subscripts denote the projection of $v$ to the corresponding subspaces). Observe that the projection of $v$ on $\fr{g}_j$ is given by $\pi_j(v)=\pi_j(v_{\fr{s}^{\prime}})+v_{\fr{p}_j}$, and write 

\begin{equation}\label{vpj}v_{\fr{p}_j}=v_{\fr{t}_j^{\prime}}+\sum_{\alpha\in R_{K_j}^+}X_{\alpha},  \ \ X_{\alpha}\in \fr{m}^{\alpha},\end{equation}

\noindent according to the decomposition \eqref{pij} of $\fr{p}_j$. Then

\begin{equation}\label{new}\pi_j(v)=\pi_j(v_{\fr{s}^{\prime}})+v_{\fr{t}_j^{\prime}}+\sum_{\alpha\in R_{K_j}^+}X_{\alpha},  \ \ X_{\alpha}\in \fr{m}^{\alpha}.
\end{equation}

\noindent In view of Equation \eqref{vpj}, our assumption that $v_{\fr{p}_j}\neq 0$ implies that at least one of the following cases is true:\\

\begin{center}  Case I) \ \ $\sum_{\alpha\in R_{K_j}^+}X_{\alpha}\neq 0$. \ \ \ \  Case II) \  $\sum_{\alpha\in R_{K_j}^+}X_{\alpha}=0$ and $v_{\fr{t}_j^{\prime}}\neq 0$.\end{center}

\noindent We will prove that in any of the above cases, the eigenvalue $\lambda$ of $v$ is equal to $\lambda_j$, which will conclude the proof of the proposition. 

For Case I), our goal is to apply Lemma \ref{EigenEq} for suitable subspaces in order to conclude that $\lambda=\lambda_j$.  To this end, let $\fr{m}_{\lambda}\subseteq \fr{p}$ be the eigenspace of $\left.A\right|_{\fr{p}}$ corresponding to the eigenvalue $\lambda$.  By the assumption for Case I), at least one vector $X_{\alpha_0}$, $\alpha_0\in R_{K_j}^+$, is non-zero.  By Lemma \ref{SumExis}, there exists a root $\beta\in R_{M_j}$ such that $\alpha_0+\beta \in R_{M_j}$. Set $\gamma_1:=\alpha_0+|\beta|$, $\gamma_2:=|\alpha_0-|\beta||$ and let $X_{|\beta|}$ be a non-zero vector in $\fr{m}^{|\beta|}\subset \fr{q}_j$.  Since $\alpha_0+\beta \in R_{M_j}$, at least one of the covectors $\gamma_1,\gamma_2$ is a root, and if it is a root then it is a positive one.  Then along with relation \eqref{ReRootBr} and the discussion following it, we deduce that $[X_{\alpha_0},X_{|\beta|}]$ is a non-zero vector in $\fr{m}^{\gamma_1}+\fr{m}^{\gamma_2}$.  We claim the following.

\begin{claim}\label{ClaimRoot}There does not exist a root $\alpha^{\prime}_0\in R_{K_j}^+$, different from $\alpha_0$, such that $|\alpha_0^{\prime}\pm|\beta||=\gamma_1$ or $|\alpha_0^{\prime}\pm|\beta||=\gamma_2$.
\end{claim}

\noindent \emph{Proof of Claim \ref{ClaimRoot}.} If there exists a root $\alpha_0^{\prime}\in R_{K_j}^+$ such that $|\alpha_0^{\prime}\pm|\beta||=\gamma_1$ or $|\alpha_0^{\prime}\pm|\beta||=\gamma_2$, then one of the following cases is true: 
  
  \begin{equation*} \makebox{a)} \ \alpha_0^{\prime}=\alpha_0, \ \ \makebox{\ b)} \ \alpha_0^{\prime}=-\alpha_0, \ \ \makebox{\ c)}  \ |\alpha_0^{\prime}\pm\alpha_0|=2|\beta|.\end{equation*}
  
  \noindent  Since $\alpha_0^{\prime},\alpha_0\in R_j^+$, case b) is rejected. Case c) is also rejected because the covector $|\alpha_0^{\prime}\pm\alpha_0|$ has zero coefficients in all simple roots in $\Pi_{M_j}$ while the covector $2|\beta|$ has a non-zero coefficient in at least one of the simple roots in $\Pi_{M_j}$. Therefore, $\alpha_0^{\prime}=\alpha_0$.\qed\\
  
   Taking into account Claim \ref{ClaimRoot} and relations \eqref{ReRootBr}, we conclude that the projection of the vector $[\sum_{\alpha\in R_{K_j}^+}X_{\alpha},X_{|\beta|}]$ on $\fr{m}^{\gamma_1}+\fr{m}^{\gamma_2}$ coincides with the projection of the vector $[X_{\alpha_0},X_{|\beta|}]$ on $\fr{m}^{\gamma_1}+\fr{m}^{\gamma_2}$ and thus it is non-zero.  On the other hand, by relation \eqref{TPJCC}, the space $\fr{t_j}^{\prime}$ lies in the Cartan subalgebra $\fr{h}_j$ of $\fr{g}_j$.  Moreover, $\fr{s}^{\prime}\subset \fr{f}\subset \fr{h}$.  Therefore, the vector $\pi_j(v_{\fr{s}^{\prime}})+v_{\fr{t}_j^{\prime}}$ lies in $\fr{h}_j$. Hence, by relation \eqref{Xa} we obtain 
     
     \begin{equation}\label{anter}[\pi_j(v_{\fr{s}^{\prime}})+v_{\fr{t}_j^{\prime}},X_{|\beta|}]\subseteq [\fr{h}_j,\fr{m}^{|\beta|}]\subseteq \fr{m}^{|\beta|}.\end{equation}
     \noindent  Along with Equation \eqref{new}, the fact that $[\sum_{\alpha\in R_{K_j}^+}X_{\alpha},X_{|\beta|}]$ has non-zero projection on $\fr{m}^{\gamma_1}+\fr{m}^{\gamma_2}$ and the fact that $|\beta|\neq \gamma_1,\gamma_2$, relation \eqref{anter} implies that the vector $[\pi_j(v),X_{|\beta|}]$ has non-zero projection on $\fr{m}^{\gamma_1}+\fr{m}^{\gamma_2}$.  Moreover, since $X_{|\beta|}\in \fr{q}_j$, we have $[\pi_j(v),X_{|\beta|}]=\pi_j([v, X_{|\beta|}])$. Therefore, the vector $[v,X_{|\beta|}]$ has non-zero projection on $\fr{m}^{\gamma_1}+\fr{m}^{\gamma_2}$.  Along with the facts that $v\in \fr{m}_{\lambda}$, $X_{|\beta|}\in \fr{m}^{|\beta|}$ and $(\fr{m}^{\gamma_1}+\fr{m}^{\gamma_2})\cap (\fr{m}_{\lambda}\oplus \fr{m}^{|\beta|})=\{0\}$, the last assertion implies that\\ 
     
  \noindent   (i) the space $[\fr{m}_{\lambda},\fr{m}^{|\beta|}]$ has non-zero projection on $(\fr{m}_{\lambda}\oplus\fr{m}^{|\beta|})^{\bot}$.\\    

On the other hand, both spaces $\fr{m}_{\lambda}\subseteq \fr{p}$ and $\fr{m}^{|\beta|}\subseteq \fr{q}_j$ are $\operatorname{ad}_{\fr{s}}$-invariant; Indeed, because of the $\operatorname{ad}_{\fr{s}}$-equivariance of $\left.A\right|_{\fr{p}}$, any eigenspace $\fr{m}_{\lambda}$ of $\left.A\right|_{\fr{p}}$ is $\operatorname{ad}_{\fr{s}}$-invariant, while $[\fr{s},\fr{m}^{|\beta|}]=[\pi_j(\fr{s}),\fr{m}^{|\beta|}]\subseteq [\fr{h}_j,\fr{m}^{|\beta|}]\subseteq \fr{m}^{|\beta|}$.  Moreover, \\

\noindent (ii) $\left.A\right|_{\fr{m}_{\lambda}}=\lambda\operatorname{Id}$ and $\left.A\right|_{\fr{m}^{|\beta|}}=\lambda_j\operatorname{Id}$ (the latter equation follows from Proposition \ref{SimpFlag}). \\

 \noindent Taking into account (i) and (ii) and applying Lemma \ref{EigenEq} we conclude that $\lambda=\lambda_j$ for Case I).\\\\

For Case II), assume that $\sum_{\alpha\in R_{K_j}^+}X_{\alpha}=0$ and $v_{\fr{t}^{\prime}_j}\neq 0$. Equation \eqref{new} implies that

\begin{equation}\label{Tzak1}\pi_j(v)=\pi_j(v_{\fr{s}^{\prime}})+v_{\fr{t}^{\prime}_j}.\end{equation}

\noindent Since Case I) implies that $\lambda=\lambda_j$, applying Lemma \ref{LinAlg} for $V=\fr{p}$ and $W=\sum_{\alpha\in R^+_{K_j}}\fr{m}^{\alpha}$ yields

\begin{equation}\label{Ei}\left.A\right|_{\sum_{\alpha\in R^+_{K_j}}\fr{m}^{\alpha}}=\lambda_j\operatorname{Id}.\end{equation}

\noindent Since $v_{\fr{t}^{\prime}_j}\neq 0$, there exists a root $\alpha_0\in R_{K_j}^+$ such that $\alpha_0(v_{\fr{t}^{\prime}_j})\neq 0$ for otherwise, $v_{\fr{t}^{\prime}_j}\in \fr{t}_j$ which is impossible by the orthogonality of $\fr{t}_j$ and $\fr{t}_j^{\prime}$.  On the other hand, the definition \eqref{Ef} of $\fr{f}$ and the fact that $\fr{s}^{\prime}\subset \fr{f}$ imply that $\pi_j(v_{\fr{s}^{\prime}}) \in \fr{t}_j$.  Hence, the definition \eqref{center} of $\fr{t}_j$ yields $\alpha_0(\pi_j(v_{\fr{s}^{\prime}}))=0$.  Along with Equation \eqref{Tzak1}, we conclude that

 \begin{equation}\label{T1}\alpha_0(\pi_j(v))=\alpha_0(v_{\fr{t}^{\prime}_j})\neq 0.\end{equation}

\noindent Choose a non-zero vector $X_{\alpha_0}\in \fr{m}^{\alpha_0}$.  Relation \eqref{Ei} yields

 \begin{equation}\label{Ei1}AX_{\alpha_0}=\lambda_jX_{\alpha_0}.\end{equation}
 
\noindent Since $\left.A\right|_{\fr{p}}$ defines a bi-invariant metric, Lemma \ref{Bi} implies that $[X,AX]=0$ for all $X\in \fr{p}$ and thus
 
 \begin{equation}\label{BiInvT}[\pi_j(X),\pi_j(AX)]=0 \ \ \makebox{for all} \ \ X\in \fr{p}.\end{equation}

\noindent Set $X:=v+X_{\alpha_0}$. By taking into account the fact that  $\pi_j(X_{\alpha_0})=X_{\alpha_0}$, the fact that $Av=\lambda v$, Equation \eqref{Ei1}, the fact that $\pi_j(v)\in \fr{t}_j\oplus \fr{t}_j^{\prime}=\fr{h}_j$ as well as Equation \eqref{Xa}, we deduce that Equation \eqref{BiInvT} for $X$ is equivalent to

\begin{equation*}0=(\lambda_j-\lambda)[\pi_j(v),X_{\alpha_0}]=(\lambda_j-\lambda)\alpha_0(\pi_j(v))\bar{X}_{\alpha_0}.\end{equation*}

\noindent The above equation, along with relation \eqref{T1} and the fact that $\bar{X}_{\alpha_0}\neq 0$ (given that $X_{\alpha_0}\neq 0$), implies that $\lambda=\lambda_j$.  This settles Case II).\qed \\

\noindent \emph{Proof of Proposition \ref{Endo}}   By Propositions \ref{SimpFlag} and \ref{SimIsot}, the $B$-orthogonal complement $(\fr{s}^{\prime})^{\bot}=\bigoplus_{j=1}^k(\fr{p}_j\oplus \fr{q}_j)$ of $\fr{s}^{\prime}$ in $\fr{m}$ is $A$-invariant.  Hence, by taking into account the symmetry of $A$ we obtain

\begin{equation*}B\big(A\fr{s}^{\prime},(\fr{s}^{\prime})^{\bot}\big)=B\big(\fr{s}^{\prime},A(\fr{s}^{\prime})^{\bot}\big)\subseteq B\big(\fr{s}^{\prime},(\fr{s}^{\prime})^{\bot}\big)=\{0\}.
\end{equation*}

\noindent Therefore, $A\fr{s}^{\prime}$ is $B$-orthogonal to $(\fr{s}^{\prime})^{\bot}$ and thus the space $\fr{s}^{\prime}$ is $A$-invariant.

\section{Proof of the main results}\label{proof}

\subsection{Proof of Theorem \ref{main}}

 We recall the notation of Section \ref{mainresults}. For the sufficiency part of the theorem, recall that any naturally reductive metric is a g.o. metric and the same is true for any normal metric. 
  For the necessity part, let $g$ be a g.o. metric on $G/S$ and let $A$ be the corresponding metric endomorphism satisfying $g(X,Y)_o=B(AX,Y)$, $X,Y\in \fr{m}$, where $B$ is the negative of the Killing form of $\fr{g}$ and $\fr{m}$ is the $B$-orthogonal complement of $\fr{s}$ in $\fr{g}$. We will firstly prove that $g$ is naturally reductive.  
  
  Let $S^0$ be the identity component of $S$, which is a torus in $G$.  Then $\fr{s}$ is the Lie algebra of both $S$ and $S^0$.  Therefore, the space $\fr{m}$ can be identified with $T_o(G/S^0)$, and the $\operatorname{Ad}_S$-equivariant endomorphism $A:\fr{m}\rightarrow \fr{m}$ also defines a $G$-invariant metric on $G/S^0$.  Since $A$ is a g.o. metric on $G/S$, Proposition \ref{GOCond} states that there exists a map $\xi:\fr{m}\rightarrow \fr{s}$ such that 

\begin{equation}\label{StarRel}[\xi(X)+X,AX]=0 \ \ \makebox{for all} \ \ X\in \fr{m},
\end{equation}
  
\noindent and thus $A$ also defines a g.o. metric on $G/S^0$.  By Theorem \ref{NecForm}, $A$ has the form \eqref{NecForm1}.  We will initially use this form to derive an equivalent condition to Equation \eqref{StarRel}. 

In view of the decomposition $\fr{g}=\fr{g}_1\oplus \cdots \oplus \fr{g}_k$ and the fact that the projections $\pi_j:\fr{g}\rightarrow \fr{g}_j$ are homomorphisms, Equation \eqref{StarRel} is equivalent to 

\begin{equation}\label{StarRel1}[\pi_{j}(\xi(X)+X),\pi_j(AX)]=0 \ \ \makebox{for all}\ \ X\in \fr{m}\ \ \makebox{and} \ \  j=1,\dots, k.
\end{equation}

\noindent Consider the decomposition $\fr{m}=\fr{s}^{\prime}\oplus (\fr{p}_1\oplus \fr{q}_1)\oplus \cdots \oplus (\fr{p}_k\oplus \fr{q}_k)$.  For $X\in \fr{m}$, write 

\begin{equation*}X=X_{\fr{s}^{\prime}}+(X_{\fr{p}_1}+X_{\fr{q}_1})+\cdots +(X_{\fr{p}_k}+X_{\fr{q}_k}),
\end{equation*}

\noindent according to the above decomposition, and observe that 

\begin{equation}\label{PijX}\pi_j(X)=\pi_j(X_{\fr{s}^{\prime}})+ X_{\fr{p}_j}+X_{\fr{q}_j} \ \ \makebox{and} \ \ \pi_j(AX)=\pi_j(AX_{\fr{s}^{\prime}})+ \lambda_j(X_{\fr{p}_j}+X_{\fr{q}_j}),
\end{equation}

\noindent where the latter equation follows from \eqref{NecForm1}. In view of the decomposition of $\fr{q}_j$ in \eqref{koin}, write

\begin{equation}\label{koin1}
X_{\fr{q}_j}=\sum_{\alpha\in R_{M_j}^+}X_{\alpha}, \ \ X_{\alpha}\in \fr{m}^{\alpha}.
\end{equation}

\noindent Moreover, the following equations are valid.

\begin{eqnarray}\label{Centr1}&&\big[\pi_j(\xi(X)),X_{\fr{p}_j}\big]=0, \ \ \ \ \ \ \ \ \ \ \big[\pi_j(X_{\fr{s}^{\prime}}),X_{\fr{p}_j}\big]=0, \ \ \ \ \ \ \ \  \big[\pi_j(AX_{\fr{s}^{\prime}}),X_{\fr{p}_j}\big]=0 \ \ \makebox{and} \\ \label{Centr2}&&\big[\pi_j(\xi(X)),\pi_j(AX_{\fr{s}^{\prime}})\big]=0, \ \ \ \big[\pi_j(X_{\fr{s}^{\prime}}),\pi_j(AX_{\fr{s}^{\prime}})\big]=0.\end{eqnarray}

\noindent The first equation follows from the facts that $X_{\fr{p}_j}\in \fr{k}_j$, $\xi(X)\in \fr{s}$ and $\pi_j(\fr{s})$ is contained in the center $\fr{t}_j$ of $\fr{k}_j$ (Lemma \ref{LIC}).  The second equation follows also from the fact that $\pi_j(X_{\fr{s}^{\prime}})\in \fr{t}_j$, given that $\fr{s}^{\prime}\subset \fr{f}=\fr{t}_1\oplus \cdots \oplus \fr{t}_k$. The third equation is true for the same reason, and by taking into account that $\fr{s}^{\prime}$ is $A$-invariant (Proposition \ref{Endo}). The fourth equation is also true because $\fr{s}^{\prime}$ is $A$-invariant, and because $\big[\pi_j(\xi(X)),\pi_j(AX_{\fr{s}^{\prime}})\big]\subseteq \pi_j([\fr{s},\fr{s}^{\prime}])\subseteq \pi_j([\fr{f},\fr{f}])=\{0\}$.  Similarly, the fifth equation holds because $\pi_j([\fr{s}^{\prime},\fr{s}^{\prime}])=\{0\}$. 

By using the form \eqref{NecForm1} of $A$ and relations \eqref{PijX} - \eqref{Centr2}, Equation \eqref{StarRel1} is equivalent to 

\begin{equation}\label{StarRel2}
\sum_{\alpha\in R^+_{M_j}}[Y_j,X_{\alpha}]=0, \ \ \makebox{where} \ \ Y_j=Y_j(X):=\lambda_j\pi_j(\xi(X))+\big(\pi_j\circ (\lambda_j\operatorname{Id}-A)\big)X_{\fr{s}^{\prime}}.
\end{equation}

\noindent Since $\pi_j(\xi(X))\in \fr{t}_j$ and $(\pi_j\circ A)(\fr{s}^{\prime})\subseteq \pi_j(\fr{s}^{\prime})\subset \fr{t}_j$, the vector $Y_j$ lies in $\fr{t}_j$, and thus it lies in the Cartan subalgebra $\fr{h}_j$ of $\fr{g}_j$.  Hence, by virtue of Equation \eqref{Xa}, Equation \eqref{StarRel2} is equivalent to $\sum_{\alpha\in R^+_{M_j}}\alpha(Y_j)\bar{X}_{\alpha}=0$.  In summary, we have arrived to the following conclusion.

\begin{lemma}\label{Equiv}
The following are equivalent:\\

\noindent \emph{(i)} The endomorphism $A$ given by Equation \eqref{NecForm1} defines a g.o. metric on $G/S$.\\

\noindent \emph{(ii)} There exists a map $\xi:\fr{m}\rightarrow \fr{s}$ such that all $X\in \fr{m}$ satisfy

\begin{equation*}\label{axe}[\xi(X)+X,AX]=0.\end{equation*}

\noindent \emph{(iii)} There exists a map $\xi:\fr{m}\rightarrow \fr{s}$ such that all vectors $X\in \fr{m}$ satisfy 

\begin{equation}\label{StarRel3}
\sum_{\alpha\in R^+_{M_j}}\alpha\big(Y_j(X)\big)\bar{X}_{\alpha}=0,  \ \ j=1,\dots ,k,
\end{equation}

\noindent where $\sum_{\alpha\in R^+_{M_j}}X_{\alpha}$ is the projection $X_{\fr{q}_j}$ of $X$ on $\fr{q}_j$ and 

\begin{equation}\label{psit}Y_j(X)=\pi_j\bigg(\lambda_j\xi(X)+(\lambda_j\operatorname{Id}-A)(X_{\fr{s}^{\prime}})\bigg), \ \ j=1,\dots ,k,\end{equation}

\noindent where $X_{\fr{s}^{\prime}}$ is the projection of $X$ on $\fr{s}^{\prime}$.
\end{lemma}

Continuing with the main proof, since $A$ defines a g.o. metric, there exists a map $\xi:\fr{m}\rightarrow \fr{s}$ such that Equation \eqref{StarRel3} is satisfied.  To prove that $A$ is naturally reductive, by virtue of Proposition \ref{SZ2} and the equivalence of parts (ii) and (iii) in Lemma \ref{Equiv}, it suffices to find an $\operatorname{Ad}_S$-equivariant linear map $\widetilde{\xi}:\fr{m}\rightarrow \fr{s}$, possibly different from $\xi$, such that Equation \eqref{StarRel3} remains true for $Y_j(X)=\pi_j\bigg(\lambda_j\widetilde{\xi}(X)+(\lambda_j\operatorname{Id}-A)(X_{\fr{s}^{\prime}})\bigg)$.

 In view of the decomposition $\fr{m}=\fr{s}^{\prime}\oplus \fr{p}\oplus \fr{q}$, write $X=X\big(X_{\fr{s}^{\prime}},X_{\fr{p}},X_{\fr{q}}\big)$, where $X_{\fr{p}}=X_{\fr{p}_1}+\cdots +X_{\fr{p}_k}$ and $X_{\fr{q}}=X_{\fr{q}_1}+\cdots +X_{\fr{q}_k}$.  Moreover, write

\begin{equation*}\xi(X)=\xi(X_{\fr{s}^{\prime}},X_{\fr{p}},X_{\fr{q}}) \ \ \makebox{and} \ \ Y_j(X)=Y_j(X_{\fr{s}^{\prime}},X_{\fr{p}},X_{\fr{q}}). \end{equation*}

\noindent  We will choose the new map $\widetilde{\xi}:\fr{m}\rightarrow \fr{s}$ in such a way that it will be independent from $X_{\fr{p}}$ and $X_{\fr{q}}$, and will depend only on $X_{\fr{s}^{\prime}}$. To this end, fix the vector $X_{\fr{p}}^0:=0$ and vectors $X^0_{\fr{q}_j}:=\sum_{\alpha\in R^+_{M_j}}X_{\alpha}$, $j=1,\dots ,k$, such that $X_{\alpha}\neq 0$ (and thus $\bar{X}_{\alpha}\neq 0$) for all $\alpha\in R^+_{M_j}$.  Set $X_{\fr{q}}^0:=\sum_{j=1}^kX^0_{\fr{q}_j}$ and

\begin{equation*}X^0:=X\big(X_{\fr{s}^{\prime}},X^0_{\fr{p}},X^0_{\fr{q}}\big).\end{equation*}

\noindent  Since $X_{\fr{p}}^0,X_{\fr{q}}^0$ are fixed, $X^0$ depends only on $X_{\fr{s}^{\prime}}$. Since $A$ is a g.o. metric and since the vectors $\bar{X}_{\alpha}$, $\alpha\in R_{M_j}$, are linearly independent and non-zero for all $\alpha\in R_{M_j}$ (due to the choice of $X^0_{\fr{q}_j}$), Equation \eqref{StarRel3} for $X=X^0$ yields

 \begin{equation*}\alpha \big(Y_j(X^0)\big)=0,\end{equation*}
 
 \noindent for all $\alpha\in R^+_{M_j}$ and for all $j=1,\dots,k$. Hence, by virtue of decomposition \eqref{koin} for $\fr{q}_j$ and relation \eqref{Xa}, we obtain $[Y_j(X^0),\fr{q}_j]=\{0\}$.  On the other hand, as discussed above, each vector $Y_j(X^0)$, $j=1,\dots,k$, lies in the center $\fr{t}_j$ of $\fr{k}_j$.  We conclude that $Y_j(X^0)$ lies in the center of $\fr{g}_j=\fr{k}_j\oplus \fr{q}_j$.  Since $\fr{g}_j$ is simple, $Y_j(X^0)=0$ for all $j=1,\dots,k$.  Therefore, Equation \eqref{psit} for $X=X^0$ yields 

\begin{equation}\label{LR}
(\pi_j\circ\xi)(X^0)=\frac{1}{\lambda_j}\big(\pi_j\circ (A-\left.\lambda_j\operatorname{Id}\right.)\big)(X_{\fr{s}^{\prime}}).
\end{equation}

We introduce the map $\widetilde{\xi}:\fr{m}\rightarrow \fr{s}$, defined by 

\begin{equation}\label{XiDef}\widetilde{\xi}(X):=\xi(X^0) \ \ \makebox{for} \ \ X=X\big(X_{\fr{s}^{\prime}},X_{\fr{p}},X_{\fr{q}}\big)\in \fr{m}.\end{equation}

 \noindent The map $\widetilde{\xi}$ is well-defined because the image of ${\xi}$ lies in $\fr{s}$.  Moreover, Equation \eqref{LR} yields

\begin{equation}\label{CDM}\widetilde{\xi}(X)=\sum_{j=1}^k (\pi_j\circ \xi)(X^0)=\sum_{j=1}^k\frac{1}{\lambda_j}\big(\pi_j\circ (A-\left.\lambda_j\operatorname{Id}\right.)\big)(X_{\fr{s}^{\prime}}).
\end{equation}

 \noindent The map $\widetilde{\xi}$ satisfies the following properties:\\

\noindent 1) Since $X^0$ depends only on $X_{\fr{s}^{\prime}}$, $\widetilde{\xi}$ also depends only on $X_{\fr{s}^{\prime}}$.  Therefore $\widetilde{\xi}(X)=\widetilde{\xi}(X_{\fr{s}^{\prime}})$.\\

\noindent 2) By Equation \eqref{CDM}, $\widetilde{\xi}$ is linear on $\fr{s}^{\prime}$. Hence, by property 1) $\widetilde{\xi}$ is linear on $\fr{m}$.\\

\noindent 3)  The map $\widetilde{\xi}$ is $\operatorname{Ad}_S$-equivariant (and thus $\operatorname{Ad}_{S^0}$-equivariant).\\

\noindent Indeed, $[X,\pi_j(Y)]=[\pi_j(X),\pi_j(Y)]=\pi_j([X,Y])$ for all $X,Y\in \fr{g}$.  Therefore, given that $G$ is connected, each map $\pi_j$ is $\operatorname{Ad}_G$-equivariant and thus $\operatorname{Ad}_S$-equivariant. Moreover, $A$ is $\operatorname{Ad}_S$-equivariant.  Therefore, by relation \eqref{CDM} it follows that $\widetilde{\xi}$ is $\operatorname{Ad}_S$-equivariant.\\

\noindent 4) The vector $\widetilde{\xi}(X)+X$ satisfies the equation 

\begin{equation}\label{StarRel4}[\widetilde{\xi}(X)+X,AX]=0.\end{equation}

\noindent  Indeed, by Lemma \ref{Equiv}, Equation \eqref{StarRel4} is equivalent to Equation \eqref{StarRel3}, where $\sum_{\alpha\in R^+_{M_j}}\bar{X}_{\alpha}$ is the projection $X_{\fr{q}_j}$ of $X$ on $\fr{q}_j$ and $Y_j(X)=\pi_j\bigg(\lambda_j\widetilde{\xi}(X)+(\lambda_j\operatorname{Id}-A)(X_{\fr{s}^{\prime}})\bigg)$.  On the other hand, relation \eqref{CDM} implies that $\lambda_j(\pi_j\circ\widetilde{\xi})(X)=\lambda_j(\pi_j\circ \xi)(X^0)=\big(\pi_j\circ (A-\left.\lambda_j\operatorname{Id}\right.)\big)(X_{\fr{s}^{\prime}})$ and hence $Y_j(X)=0$. Therefore, Equation \eqref{StarRel3} is trivially satisfied, which verifies Equation \eqref{StarRel4}. \\

From properties 2), 3) and 4) and Proposition \ref{SZ2}, we conclude that $A$ defines a naturally reductive metric (on both spaces $G/S$ and $G/S^0$). Hence, we obtain the first part of Theorem \ref{main}.\\

 It remains to show that any g.o. metric $g$ on $G/S$ is a normal metric. The first part of the theorem implies that $g$ is naturally reductive with respect to some reductive decomposition $\fr{g}=\fr{s}\oplus \widetilde{\fr{m}}$.  In view of Proposition \ref{SZ2} and the proof of the first part of the theorem, we have 

\begin{equation}\label{DefWm}\widetilde{\fr{m}}=\{\widetilde{\xi}(X)+X:X\in \fr{m}\},\end{equation}
 
\noindent where $\widetilde{\xi}$ is defined by \eqref{XiDef}.  We claim the following.

\begin{claim}\label{claim2}For each $j=1,\dots,k$, there exist a non-zero subspace $V_j$ of $\fr{g}_{j}$ such that for any naturally reductive metric $g$ on $G/S$, with respect to some decomposition $\fr{g}=\fr{s}\oplus \widetilde{\fr{m}}$, $V_j\subset \widetilde{\fr{m}}$.\end{claim}

\emph{Proof of Claim \ref{claim2}}. Choose a $j=1,\dots,k$, and consider the subalgebra $\fr{k}_{j}=\fr{k}\cap \fr{g}_{j}$.  If $\fr{k}_{j}=\fr{g}_{j}$ then set $V_j:=\fr{g}_{j}$.  Since $\fr{g}_{j}$ is centerless, $\fr{t}_j=\{0\}$.  Therefore, in view of relations \eqref{IsotAlg} and \eqref{pij}, we obtain $\fr{p}_{j}=\fr{k}_{j}=V_j$, and thus $V_j\subset \fr{m}$.  Along with the fact that the corresponding metric endomorphism $A$ of $g$ has the form \eqref{NecForm1}, we deduce that $\left.A\right|_{V_j}=\left.\lambda_{j}\operatorname{Id}\right.$ for some $\lambda_j>0$.  Then relation \eqref{CDM} yields $\widetilde{\xi}(V_j)=\{0\}$.  Therefore, by relation \eqref{DefWm} and the fact that $V_j\subset \fr{m}$ we conclude that $V_j\subset \widetilde{\fr{m}}$.  If $\fr{k}_{j}\subsetneq \fr{g}_{j}$, then set $V_j:=\fr{q}_{j}$, i.e. $V_j$ is the $B$-orthogonal complement of $\fr{k}_{j}$ in $\fr{g}_{j}$.  Then $V_j\subset \fr{m}$.  Moreover, the necessary form \eqref{NecForm1} of the metric endomorphism $A$ yields $\left.A\right|_{V_j}=\left.\lambda_{j}\operatorname{Id}\right.$.  Again we have $\widetilde{\xi}(V_j)=\{0\}$,  and thus $V_j\subset \widetilde{\fr{m}}$.  In any case, we conclude that there exists a non-zero space $V_j\subseteq \fr{g}_j$ such that $V_j\subset \widetilde{\fr{m}}$. Finally, we remark that the choice of $V_j$ is independent of $g$, which concludes the proof of the claim.\qed \\

\noindent Continuing with the proof of the second part of Theorem \ref{main}, $G/S$ is connected as $G$ is connected. Moreover, since $G$ is semisimple and $S$ is abelian, $G$ acts almost effectively on $G/S$ ($G$ does not contain non-discrete abelian normal subgroups).  Since $g$ is naturally reductive with respect to the decomposition $\fr{g}=\fr{s}\oplus \widetilde{\fr{m}}$, by virtue of Theorem \ref{KosThe}, the space $\widetilde{\fr{g}}:=\widetilde{\fr{m}}+[\widetilde{\fr{m}},\widetilde{\fr{m}}]$ is an ideal of $\fr{g}$ such that the corresponding subgroup $\widetilde{G}$ of $G$ acts transitively on $G/S$.  Moreover, there exists a unique $\operatorname{Ad}_{\widetilde{G}}$-invariant, symmetric, non-degenerate bilinear form $Q$ on $\widetilde{\fr{g}}$ such that 

\begin{equation}\label{karant}Q(\fr{s}\cap\widetilde{\fr{g}},\widetilde{\fr{m}})=\{0\} \ \ \makebox{and} \ \ g( X,Y )_o=Q(X,Y) \ \ \makebox{for all} \ \ X,Y\in \widetilde{\fr{m}}.\end{equation}

Since $\widetilde{\fr{g}}$ is an ideal of the semisimple algebra $\fr{g}$, it has the form $\widetilde{\fr{g}}=\fr{g}_{i_1}\oplus \cdots \oplus \fr{g}_{i_s}$, $s\leq k$.  But Claim \ref{claim2} implies that $\widetilde{\fr{m}}$ has non-zero projection on $\fr{g}_j$ for all $j=1,\dots,k$.  Along with the fact that $\widetilde{\fr{m}}\subset \widetilde{\fr{g}}$, we deduce that $s=k$ and hence $\fr{g}=\widetilde{\fr{g}}$.  Therefore, by Proposition \ref{DZp}, $Q$ has the form 

\begin{equation}\label{Yields}Q=\left.\mu_{1}B\right|_{\fr{g}_{1}\times \fr{g}_{1}}+\cdots +\left.\mu_{k}B\right|_{\fr{g}_{k}\times \fr{g}_{k}}, \ \ \mu_{j}\in \mathbb R^*.\end{equation}

\noindent 

To prove that the metric $g$ is normal, it remains to show that $Q$ is positive definite, or equivalently that $\mu_{j}>0$, $j=1,\dots, k$.  Indeed, by considering for each $j$ the subspace $V_j$ of $\widetilde{\fr{m}}$ obtained in Claim \ref{claim2} and by taking into account that $Q$ is positive definite on $\widetilde{\fr{m}}$ (relation \eqref{karant}), we deduce that $\left.Q\right|_{V_j \times V_j }$ is positive definite.  By relation \eqref{Yields}, $\left.Q\right|_{V_j \times V_j }=\left.\mu_{j}B\right|_{V_j \times V_j }$.  Therefore, $\mu_{j}>0$, which concludes the proof of Theorem \ref{main}.

 \subsection{Proof of Corollary \ref{Class}}
 By taking into account Theorem \ref{main}, Definition \ref{normal} and Proposition \ref{DZp}, we deduce that $g$ is a g.o. metric if and only if $g(\ , \ )_o$ is the restriction $\left.Q\right|_{\widetilde{\fr{m}}\times \widetilde{\fr{m}}}$ of an inner product of the form $Q=\sum_{j=1}^k\left.\mu_j B\right|_{\fr{g}_j\times \fr{g}_j}$, where $B$ is the negative of the Killing form of $\fr{g}$ and $\widetilde{\fr{m}}$ is the $Q$-orthogonal complement of $\fr{s}$ in $\fr{g}$.  It remains to show that if $Q^{\prime}=\sum_{j=1}^k\mu^{\prime}_j \left.B\right|_{\fr{g}_j\times \fr{g}_j}$ is an $\operatorname{Ad}$-invariant inner product on $\fr{g}$ which is different from $Q$ up to homothety and $\widetilde{\fr{m}}^{\prime}$ is the $Q^{\prime}$-orthogonal complement of $\fr{s}$ in $\fr{g}$, then $\left.Q\right|_{\widetilde{\fr{m}}\times \widetilde{\fr{m}}}$ and $\left.Q^{\prime}\right|_{\widetilde{\fr{m}}^{\prime}\times \widetilde{\fr{m}}^{\prime}}$ define different metrics on $G/S$ up to homothety.  Since the products $Q$ and $Q^{\prime}$ are different up to homothety, then $G$ is non-simple and we may assume (after normalizing the metrics) that $\mu_1=1$ and that there exists an index $j_0=2,\dots,k$ such that $\mu_{j_0}\neq \mu_{j_0}^{\prime}$. Choose the non-zero spaces $V_1,V_{j_0}$, obtained in Claim \ref{claim2}.  Then $V_1\subset \fr{g}_1$, $V_{j_0}\subset \fr{g}_{j_0}$, and $V_1,V_{j_0}\subset \widetilde{\fr{m}},\widetilde{\fr{m}}^{\prime}$.  Therefore, $\left.Q\right|_{V_1\times V_1}=\left.B\right|_{V_1\times V_1}=\left.Q^{\prime}\right|_{V_1\times V_1}$, while $\left.Q\right|_{V_{j_0}\times V_{j_0}}=\left.\mu_{j_0}B\right|_{V_{j_0}\times V_{j_0}}\neq \mu_{j_0}^{\prime}\left.B\right|_{V_{j_0}\times V_{j_0}}=\left.Q^{\prime}\right|_{V_{j_0}\times V_{j_0}}$.  Hence, the metrics $\left.Q\right|_{\widetilde{\fr{m}}\times \widetilde{\fr{m}}}$, $\left.Q^{\prime}\right|_{\widetilde{\fr{m}}^{\prime}\times \widetilde{\fr{m}}^{\prime}}$ are different up to homothety.

\end{document}